\documentclass[12pt,reqno]{amsart}
\usepackage{geometry}                		
\usepackage{graphicx}				
\usepackage{epstopdf}
\usepackage{subfigure}				
\usepackage{amssymb}

\usepackage{datetime}
\usepackage{currfile} 

\usepackage{color}

\usepackage[toc,page]{appendix}

\newtheorem{thm}{Theorem}[section]

\newtheorem{lem}[thm]{Lemma}

\theoremstyle{definition}

\theoremstyle{remark}
\newtheorem{rem}{Remark}[section]

\numberwithin{equation}{section}

\numberwithin{equation}{section}

\newcommand{\red}[1]{\textcolor{black}{#1}}
\newcommand{\blue}[1]{\textcolor{black}{#1}}
	
\begin{document}

\title[] 
      {An invariant-region-preserving limiter for DG schemes to isentropic Euler equations}      
\author{
Yi Jiang and Hailiang Liu
}
\address{Iowa State University, Mathematics Department, Ames, IA 50011} \email{yjiang1@iastate.edu}
\address{Iowa State University, Mathematics Department, Ames, IA 50011} \email{hliu@iastate.edu}


\thanks{This work was supported by the National Science Foundation under Grant DMS1312636 and by NSF Grant RNMS (Ki-Net) 1107291.
}


\begin{abstract} 

In this paper, we introduce an invariant-region-preserving (IRP) limiter for the p-system and the corresponding viscous p-system, both of which share the same invariant region. Rigorous analysis is presented to show that for smooth solutions the order of approximation accuracy is not destroyed by the IRP limiter, provided the cell average stays away from the boundary of the invariant region.  Moreover, this limiter is explicit, and easy for computer implementation.
 A generic algorithm incorporating the IRP limiter is presented for high order finite volume type schemes as long as  
 the evolved cell average of the underlying scheme stays strictly within the invariant region. 
  For any high order  discontinuous Galerkin (DG) scheme to the p-system, sufficient conditions are obtained for cell averages to stay in the invariant region.  For the viscous p-system, we design both second and third order IRP DG schemes.  Numerical experiments are provided to test the proven properties of the IRP limiter and the performance of IRP DG schemes.
\end{abstract}

\maketitle

\section{Introduction} 
\label{sec:1}
\red{We are interested in invariant-region-preserving (IRP) high order numerical approximations of solutions to systems of hyperbolic conservation laws 
\begin{equation}\label{cptfm+}
\mathbf{w}_t+F(\mathbf{w})_x=0,  \quad x\in \mathbb{R}, \quad t>0
\end{equation}
with the unknown vector $\mathbf{w}\in \mathbb{R}^m (m\geq 2)$ and the flux function $F(\mathbf{w})\in \mathbb{R}^m$, subject to initial data $ \mathbf{w}(x, 0)=\mathbf{w}_0(x)$. An invariant region $\Sigma$ to this system is a convex open set in phase space $\mathbb{R}^m$ so that if the initial data is in the region $\Sigma$, then the solution will remain in $\Sigma$. {It is desirable to construct high order numerical schemes solving (\ref{cptfm+}) that can preserve the whole invariant region $\Sigma$, which is in general a difficult problem.}
}  

\red{For scalar conservation equations the notion of invariant region is closely related to the maximum principle. In the case of nonlinear systems, the notion of maximum principle does not apply and must be replaced by the notion of invariant region. There are models that feature known invariant regions. For example, the invariant region of  $2\times 2$ ($m=2$) systems of  hyperbolic conservation laws can be described by two Riemann invariants. In this work, we focus on a model system, i.e., the p-system and its viscous counterpart, although the specific form of the system is not essential for the IRP approach. Other one-dimensional hyperbolic system of conservation laws can be studied along the same lines as long as it admits a convex invariant region. 
}

The initial value problem (IVP) for the p-system is given by 
\begin{equation}\label{ps}
\begin{aligned}
& v_t-u_x=0,  \quad x\in \mathbb{R}, \quad t>0\\
& u_t+p(v)_x=0, \\
& v=v_0 >0, \quad u=u_0, \quad x\in  \mathbb{R}, \; t=0,
\end{aligned}
\end{equation}
where  $p:  \mathbb{R}^+ \to \mathbb{R}^+$ satisfies  $p'<0$ and $p''>0$. The choice of $p=kv^{-\gamma}$ with the adiabatic gas constant $\gamma>1$ and positive constant $k>0$ leads to the isentropic (=constant entropy) gas dynamic equations. These equations represent the conservation of mass and momentum, where $v$ denotes the specific volume, $v=\frac{1}{\rho}$ and $\rho$ is the density,  $u$ denotes the velocity, see e.g. \cite{Sm94}.  This system and the counterpart in Eulerian coordinates are called compressible Euler equations. 

The entropy solution of the $p-$system can be realized \blue{as the limit of the following diffusive system} 
\begin{equation}\label{pes}
\begin{aligned}
& v_t-u_x=\epsilon v_{xx},  \quad x\in \mathbb{R}, \quad t>0, \\
& u_t+p(v)_x=\epsilon u_{xx}, \\
& v=v_0 >0, \quad u=u_0, \quad x\in  \mathbb{R}.
\end{aligned}
\end{equation}
It is well known that these two systems share  a common invariant region $\Sigma$, which is a convex open set in the phase space and expressed by two Riemann invariants of the p-system.
For more general discussion on invariant regions, we refer to \cite{Sm94, Hoff}. 

\red{The main objective of this paper is to present both an explicit IRP limiter and several high order IRP schemes to solve the above two systems. The constructed high order numerical schemes can thus preserve the whole invariant region $\Sigma$. }

\blue{IRP high order methods have seldom been studied from a numerical point of view and the main goal of the present article is to introduce the required tools and explain how the standard approach for preserving the maximum principle of scalar conservation laws should be modified.}  While the presentation is given for this particular model,  it can easily be carried over to other systems equipped with a convex invariant region. For example, the one-dimensional shallow water system, and the isentropic compressible Euler system in Eulerian coordinates (see Section 2 for details of their respective invariant regions).   On the other hand, it does not seem easy to extend the analysis to multi-dimensional setting. One technical difficulty is that the invariant region determined by the Riemann invariants does not apply to the multi-dimensional case.

\red{We should also point out the global existence of the right physical solutions to compressible Euler equations is a formidable open problem in general, and the underlying system has been of much academic interest, since the isentropic case allows for a simplified mathematical system with a rich structure. Indeed the global existence has been well established following the ideas of Lax in \cite{La73} on entropy restrictions upon bounded solutions and the program initiated by Diperna \cite{Di83a, Di83b}, extended and justified by Chen \cite{Ch}, and completed by Lions et al \cite{LPS96} for all $\gamma \in (0, \infty)$, using tools of the kinetic formulation \cite{LPT94}. 
}


Enforcing the IRP property numerically should help damp oscillations in numerical solutions, as evidenced by the maximum-principle-preserving high order finite volume schemes developed in \cite{LiuOsher1996, JiangTadmor} for scalar conservation laws. A key development along this line is the work by Zhang and Shu \cite{ZhangShu2010a}, where the authors constructed a maximum-principle-preserving limiter for finite volume (FV) or discontinuous Galerkin (DG) schemes of arbitrary order for solving scalar conservation laws. Their limiter is proved to maintain the original high order accuracy of the numerical approximation (see \cite{Zh17}).
 
For nonlinear systems of conservation laws in several space variables, invariant regions under the Lax-Friedrich schemes were studied by Frid in \cite{Fr95,Fr01}.  A recent study of invariant regions for general nonlinear hyperbolic systems using the first order continuous finite elements was given by Guermond and Popov in \cite{GP16}.  However, there has been little study on the preservation of the invariant region for systems by high order numerical schemes. In the literature, instead of the whole invariant region, positivity of some physical quantities are usually considered. Positivity-preserving finite volume schemes for Euler equations are constructed in  \cite{PerthameShu1996} based on positivity-preserving properties by the corresponding first order schemes for both density and pressure. Following \cite{PerthameShu1996}, positivity-preserving high order DG schemes for compressible Euler equations were first introduced in  \cite{ZhangShu2010b}, where the limiter introduced in \cite{ZhangShu2010a} is generalized.  
We also refer to \cite{XingZhangShu2010, XingShu2011, WangZhangShuNing2012, ZhangXiaShu2012} for more related works about  hyperbolic systems of conservation laws including Euler equations, reactive Euler equations, and shallow water equations.
 A more closely related development is the work by Zhang and Shu \cite{ZhangShu2012}, where the authors introduced a minimum-entropy-principle-preserving limiter for high order schemes to the compressible Euler equation, while the limiter is implicit in the sense that the limiter parameter is solved by Newton's iteration. It is explained there how the high order accuracy can be maintained for generic smooth solutions.  An explicit IRP limiter for compressible Euler equations was recently proposed by us in \cite{JL17}. \red{The main distinction between the limiter in \cite{ZhangShu2012} 
 and that in \cite{JL17}  is that we give an explicit formula, with a single uniform scaling parameter for the whole vector solution polynomials.
This is particularly relevant at reducing computational costs in numerical implementations. 
 }
 

In this paper, we construct an IRP limiter to preserve the whole invariant region $\Sigma$ of the p-system and the viscous p-system. The cell average of numerical approximation polynomials is used as a reference to pull each cell polynomial into the invariant region. The limiter parameter is given explicitly, which is made possible by using the convexity and concavity of the two Riemann invariants, respectively. Such explicit form is quite convenient  for computer implementation. Moreover, rigorous analysis is presented to prove that for smooth solutions the high order of accuracy is not destroyed by the limiter in general cases. By general cases we mean that in some rare cases accuracy deterioration can still  happen, as shown by example in Appendix B.

Furthermore, we present a generic algorithm incorporating this IRP limiter for high order finite volume type schemes as long as the evolved cell average by the underlying scheme stays strictly in the invariant region. This is true for first order schemes under  proper CFL conditions (see, for instance, Theorem \ref{thm:1stFV}). For high order schemes, this may hold true provided  solution values on a set of test points (called test set hereafter) stay within the invariant region, in addition to the needed CFL condition. Indeed we are able to obtain such CFL condition and test set for any high order DG schemes to the p-system.

\red{For the viscous p-system we present both second and third order DDG schemes, using the DDG diffusive flux proposed in \cite{LY10},  and prove that the cell average remains strictly within the invariant region under some sufficient conditions. In order to ensure that the diffusive contribution lies strictly in $\Sigma_0$, the interior of $\Sigma$, we use the positivity-preserving results proved by Liu and Yu in \cite{LYH14} for linear Fokker-Planck equations. 
}

The paper is organized as follows. We first review the concept of invariant region for the p-system and show that the averaging operator is a contraction in Section 2. In Section 3  we construct an explicit IRP limiter and prove that the high order of accuracy is not destroyed by such limiter in general cases. Accordingly, a generic IRP algorithm is presented for numerical implementations. In Section 4, we discuss the IRP property for cell averages of any high order finite volume type scheme for the p-system, and second and third order DG schemes for the viscous p-system.  Sufficient conditions include a CFL condition and a test set for each particular scheme with forward Euler time-discretization.  In Section 5, we present extensive numerical experiments to test the desired properties of the IRP limiter and the performance of the IRP DG schemes. In addition, the convergence of the viscous profiles to the entropy solution is illustrated from a numerical point of view.  Concluding remarks are given in Section 6.



\section{Invariant Region and Averaging}
\label{sec:2}
The p-system is strictly hyperbolic and admits two Riemann invariants
\begin{align*}
r=u-\int ^v_m\sqrt{-p'(\xi)}d\xi, \quad s=u+\int ^v_m\sqrt{-p'(\xi)}d\xi, \quad m=\inf _xv_0(x)>0,
\end{align*}
which, associated with the eigenvalues $\pm \sqrt{-p'(v)}$, satisfy
\begin{align*}
\begin{cases}
r_t+\sqrt{-p'(v)}r_x=0, \\
s_t-\sqrt{-p'(v)}s_x=0.
\end{cases}
\end{align*}
In addition we assume that the pressure function $p$ also satisfies
\begin{align}\label{p0}
\int ^v_0\sqrt{-p'(\xi)}d\xi =\infty,\quad \forall v>0.
\end{align}
This condition is met for $p(v)=v^{-\gamma}$ with $\gamma > 1$.  Set
\begin{align*}
g(v)=\int ^{v}_m\sqrt{-p'(\xi)}d\xi, 
\end{align*}
then the assumptions on the pressure implies that $g$ is increasing, concave and tends asymptotically towards $v=0$, i.e., 
\begin{align*}
g'(v)=\sqrt{-p'(v)}>0, \quad    g''(v)=\frac{-p''(v)}{2\sqrt{-p'(v)}}<0, \quad g(0)=-\infty.
\end{align*}
This shows that 
\begin{align}\label{rs}
 r =u-g(v) \; \hbox{is convex, and}\; \;  s=u+g(v)\;\; \text{ is concave},
\end{align}
and two level set curves $r=r_0$ and $s=s_0$ must intersect. We now fix
\begin{align}\label{r0s0}
r_0=\sup _x{r(v_0(x), u_0(x))}, \quad s_0=\inf _x(s(v_0(x), u_0(x)))
\end{align}
and define
\begin{align*}
\Sigma=\{(v,u)^\top |\quad  r\leq r_0, s\geq s_0\},
\end{align*}
which is known as the invariant region of the p-system, see \cite{Sm94}, in the sense that if $(v_0(x),u_0(x))^\top \in \Sigma$, then the entropy solution $(v(x,t),u(x,t))^\top \in \Sigma$.

This implies that for entropy solution, $v$ remains bounded away from $0$ and $u$ remains bounded in $L^{\infty}$-norm. But no upper bound on $v$ is available because $v=\infty$ corresponds to $\rho =0$ in Eulerian coordinates, which may actually happen.

In what follows we shall use $\Sigma _0$ to denote $\Sigma$ without the two boundary sides, i.e., 
\begin{align*}
\Sigma _0= \{(v,u)^\top \big | \quad r< r_0, s> s_0\}. 
\end{align*}

\begin{figure}[htbp]
\centering
\subfigure{\includegraphics[width=0.6\textwidth, clip]{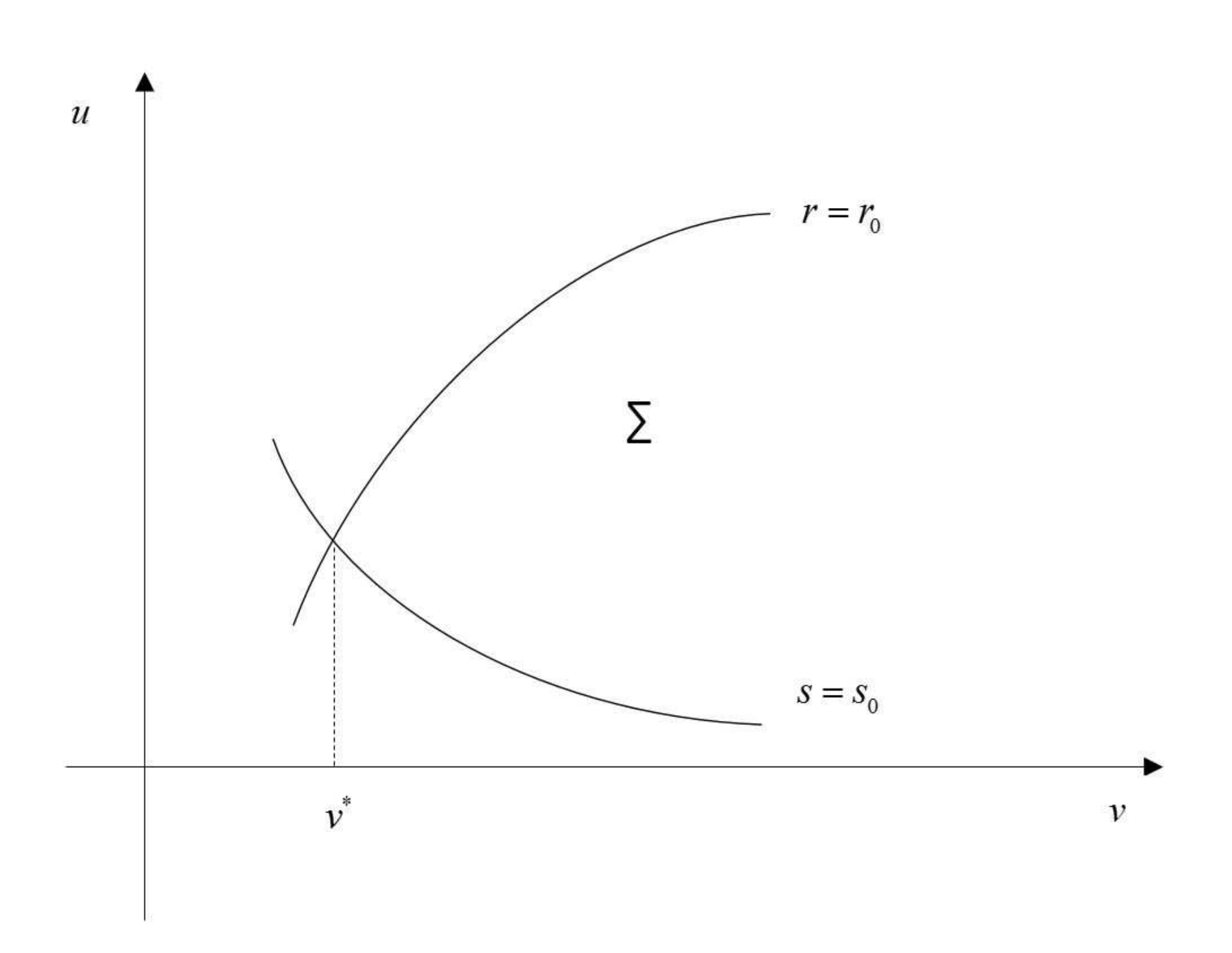}}
\caption{Invariant region for p-system}
\label{fig:IR}
\end{figure}

\begin{lem} \label{lem:1} Let $(v^*,u^*)^T$ be the intersection of upper and lower boundaries of $\Sigma$, then 
\begin{align}\label{vs}
0<v^*\leq m= \inf v_0(x).
\end{align}
\end{lem}
\begin{proof}  According to the definition of $r$ and $s$, we have
\begin{align*}
r_0=&\sup_x(u_0(x)-\int ^{v_0(x)}_m\sqrt{-p'(\xi)}d\xi)=u^*-\int ^{v^*}_m\sqrt{-p'(\xi)}d\xi,\\
s_0=&\inf _x(u_0(x)+\int ^{v_0(x)}_m\sqrt{-p'(\xi)}d\xi)=u^*+\int ^{v^*}_m\sqrt{-p'(\xi)}d\xi.
\end{align*}
Hence 
\begin{align*}
&\int ^{v^*}_m\sqrt{-p'(\xi)}d\xi=\frac{s_0-r_0}{2}\leq \int ^{v_0(x)}_m\sqrt{-p'(\xi)}d\xi, \quad \forall x.
\end{align*}
This implies that $v^* \leq v_0 (x)$ for all $x$, and  $v^*> 0$ in light of the assumption (\ref{p0}). 
\end{proof}

For any bounded interval $I$, we define the average of $\mathbf{w}(x)\doteq (v(x),u(x))^\top$ by 
\begin{align*}
\bar{\mathbf{w}}=\frac{1}{|I|}\int _I\mathbf{w}(x)dx,
\end{align*}
where $|I|$ is the measure of $I$. Such an averaging operator is a  contraction in the sense stated in the following result. 

\begin{lem}\label{lem:2}
Let $\mathbf{w}(x)=(v(x),u(x))^T$, where both $v(x)$ and $u(x)$ are non-trivial piecewise smooth functions. If $\mathbf{w}(x)\in \Sigma$ for all $x\in \mathbb{R}$, then $\bar{\mathbf{w}}\in \Sigma _0$ for any bounded interval $I$.
\end{lem}
\begin{proof}
Since $r$ is convex, according to Jensen's inequality, we have
\begin{align*}
r(\bar{\mathbf{w}})=&r\left(\frac{1}{|I|}\int _{I}\mathbf{w}(x)\right)\leq \frac{1}{|I|}\int _{I}r(\mathbf{w}(x))dx\leq r_0.
\end{align*}
If $r(\bar{\mathbf{w}})=r_0$, then $r(\mathbf{w}(x))=r_0$ for all $x$. That is,
$$
\bar{u}- g(\bar v) =u(x)- g(v(x)).
$$
By taking cell average of this relation on both sides, we have
\begin{align*}
\bar{u}- g(\bar v)=\bar u - \frac{1}{|I|}\int_{I}g(v(x))dx.
\end{align*}
This gives 
\begin{align*}
\frac{1}{|I|}\int _{I}g(v(x))dx = g(\bar{v}).
\end{align*}
Since we have
\begin{align*}
g(\bar{v}) \geq  \frac{1}{|I|}\int _{I}g(v(x))dx
\end{align*}
implied by Jensen's inequality, we conclude that $v(x)$ must be a constant, so is $u(x)=r_0 +g(v(x))$.
Therefore, $r(\bar{\mathbf{w}})<r_0$ when $\mathbf{w}(x)$ is non-trivial. The proof of $s(\bar{\mathbf{w}})>s_0$ is entirely similar.
\end{proof}
\begin{rem} From given initial data  in $\Sigma$, the above result ensures that its cell average lies strictly within $\Sigma _0$.  On the other hand, the usual piecewise $L^2$ projection, used as initial data for numerical schemes, may not lie entirely in $\Sigma$, but it has the same cell average. Such cell average is the key ingredient in the construction of our IRP limiter.   
\end{rem}

\red{
In passing we present the  invariant regions of two other equations, which can be studied along the same lines. }

\red{1.  The one dimensional isentropic gas dynamic system in Euler coordinates, i.e.
\begin{equation}\label{IGDE}
\begin{aligned}
&\rho _t+(\rho u)_x=0,\\
&(\rho u)_t+(\rho u^2+p(\rho))_x=0,
\end{aligned}
\end{equation}
with initial condition $(\rho _0(x),m_0(x))$, where $p(\rho)=\rho ^{\gamma}$, $\gamma >1$ and $m =\rho u$. The two Riemann invariants are
\begin{align*}
r=u+\frac{2\sqrt{\gamma}}{\gamma-1}\rho^{(\gamma-1)/2}, \quad s=u-\frac{2\sqrt{\gamma}}{\gamma-1}\rho^{(\gamma-1)/2}.
\end{align*}
Hence 
\begin{align}\label{s1}
\Sigma_1 =\{(\rho ,m)^{\top}|\quad r\leq r_0, \quad s\geq s_0\} 
\end{align}
is the invariant region for system (\ref{IGDE}), where  $r_0$ and $s_0$ are chosen as 
\begin{align*}
r_0=\sup _x r(\rho _0(x),m_0(x)), \quad s_0=\inf _x s(\rho _0(x),m_0(x)).
\end{align*}
}

\red{2.  The one dimensional shallow water equations with a flat bottom topography  take the form 
\begin{equation}\label{IGDE+}
\begin{aligned}
&h_t+(h u)_x=0,\\
&(h u)_t+(h u^2+\frac{g}{2}h^2)_x=0,
\end{aligned}
\end{equation}
where $h$ is the fluid height, $u$ the fluid velocity, and $g$ is the acceleration due to gravity, subject to initial condition $(h_0(x),m_0(x))$, where  $m =h u$. The two Riemann invariants are
\begin{align*}
r=u+2\sqrt{gh}, \quad s=u-2\sqrt{gh}.
\end{align*}
Hence 
\begin{align}\label{s2}
\Sigma_2 =\{(h ,m)^{\top}|\quad r\leq r_0, \quad s\geq s_0\} 
\end{align}
is the invariant region for system (\ref{IGDE+}), where  $r_0$ and $s_0$ are chosen as 
\begin{align*}
r_0=\sup _xr(h _0(x),m_0(x)), \quad s_0=\inf _x s(h_0(x),m_0(x)).
\end{align*}
In contrast to the invariant region in Figure 1, which is open and unbounded,  both $ \Sigma_1$ and $\Sigma_2$ are closed and bounded regions in terms of $(\rho, m)$ and $(h, m)$,  respectively. 
}

\section{The IRP limiter}
\label{sec:3}
\subsection{A limiter to enforce the IRP property}
Let $\mathbf{w}_h(x)=(v_h(x),u_h(x))^\top $ be a vector of polynomials of degree $k$ over an interval $I$, which is a high order approximation to the smooth function $\mathbf{w}(x)=(v(x),u(x))^\top \in \Sigma$. We assume that the average $\bar{\mathbf{w}}_h\in \Sigma _0$, but $\mathbf{w}_h(x)$ is not entirely located in $\Sigma$. We then seek a modified polynomial using the cell average as a reference  to achieve three objectives: (i) the modified polynomial preserves the cell average,
(ii) it lies entirely within $\Sigma$, and (iii) the high order of accuracy is not destroyed.

The modification is through a linear convex combination, 
which is of the form
\begin{align}\label{bb}
\tilde{\mathbf{w}}_h(x)=\theta \mathbf{w}_h(x)+(1-\theta)\bar{\mathbf{w}}_h, 
\end{align}
 where  $\theta \in (0,1] $ is defined by 
\begin{align}\label{limiter}
\theta =\min \{1,\theta _1,\theta _2\},
\end{align}
with 
\begin{align*}
\theta _1=
\frac{r_0-r(\bar{\mathbf{w}}_h)}{r_{\text{max}}-r(\bar{\mathbf{w}}_h)}
, \quad \quad 
\theta _2=
\frac{s(\bar{\mathbf{w}}_h)-s_0}{s(\bar{\mathbf{w}_h})-s_{\text{min}}}
\end{align*}
and
\begin{align}\label{aa}
r_{\max} = \max _{x\in I}r(\mathbf{w}_h(x)), \quad s_{\min}=\min _{x\in I}s(\mathbf{w}_h(x)).
\end{align}
Notice that since $\bar{\mathbf{w}}_h\in \Sigma _0$, we have $r(\bar{\mathbf{w}}_h)<r_0$ and $s(\bar{\mathbf{w}}_h)>s_0$. Also $r(\bar{\mathbf{w}}_h)<r_{\max}$ and $s(\bar{\mathbf{w}}_h)>s_{\min}$ due to the convexity of $r$ and concavity of $s$, respectively. Both $\theta _1$ and $\theta _2$ are well-defined and positive.

We proceed to show (ii) in Lemma \ref{lem:3}, and prove the accuracy-preserving property (iii) rigorously in Lemma \ref{lem:5}. 

\begin{lem}\label{lem:3}
If $\bar{\mathbf{w}}_h \in \Sigma _0$, then $\tilde{\mathbf{w}}_h(x)\in \Sigma$, $\forall x\in I$.
\end{lem}
\begin{proof} We only need to discuss the case $\theta=\theta _1$ or $\theta_2$. 
If $\theta=\theta _1$, with the convexity of $r$, we have
\begin{align*}
r(\tilde{\mathbf{w}}_h(x))\leq &\theta_1 r(\mathbf{w}_h(x))+(1-\theta_1)r(\bar{\mathbf{w}}_h)\\
= &(r_0-r(\bar{\mathbf{w}}_h))\frac{r(\mathbf{w}_h(x))-r(\bar{\mathbf{w}}_h)}{r_{\max}-r(\bar{\mathbf{w}}_h)}+r(\bar{\mathbf{w}}_h)\leq r_0.
\end{align*}
On the other hand,   from $\theta_1 \leq \theta_2$ it follows that
\begin{align}
(s(\bar{\mathbf{w}}_h)-s_{\text{min}})\theta_1 \leq s(\bar{\mathbf{w}}_h) -s_0.
\end{align}
By the concavity of $s$, we have
\begin{align*}
s(\tilde{\mathbf{w}}_h(x)) \geq & \theta _1 s(\mathbf{w}_h(x))+(1-\theta _1)s(\bar{\mathbf{w}}_h)\\
\geq & \theta _1 (s_{\text{min}}-s(\bar{\mathbf{w}}_h)) + s(\bar{\mathbf{w}}_h)\\
\geq & s_0 -s(\bar{\mathbf{w}}_h)+s(\bar{\mathbf{w}}_h)=s_0.
\end{align*}
In the case that $\theta=\theta _2$, the proof is entirely similar. 
\end{proof}

Next we show that this limiter does not destroy the accuracy. We first prepare the following lemma. We adopt the notation $\| \cdot \|_{\infty}\doteq \max _{x\in I}|\cdot |$ in the rest of the paper.

\begin{lem}\label{lem:4}
If $\|\mathbf{w}_h-\mathbf{w}\|_\infty$ is sufficiently small, we have
\begin{itemize}
\item[(i)] $\| \mathbf{w}_h\| _{\infty}\leq \| \mathbf{w}\| _{\infty}+1$;\\
\item[(ii)] $\frac{v^*}{2}\leq v_h(x)\leq \|v\|_{\infty}+1, \quad \forall x\in I$;\\
\item[(iii)] $\| \triangledown r\|_{\infty}\leq \sqrt{1-p'(v^*/2)},  \quad g'(v) \geq \sqrt{-p'(\| v\|_{\infty}+1)}$.
\end{itemize}
\end{lem}
\begin{proof}
The statements (i) and (ii) follow from the assumption and the fact that $v^*\leq v(x)\leq  \|v\|_\infty$, for all $x\in I$.
And (iii) results from the facts that $\| \triangledown r\|_2=\sqrt{1-p'(v)}$ is decreasing and $g'(v)=\sqrt{-p'(v)}$ is decreasing, with respect to $v$.
\end{proof}

\begin{lem}\label{lem:5}
The reconstructed polynomial preserves high order accuracy, i.e.
\begin{align}\label{acc}
\| \tilde{\mathbf{w}}_h-\mathbf{w}\| _{\infty} \leq C \| \mathbf{w}_h-\mathbf{w}\| _{\infty},
\end{align}
where $C>0$ depends on $\mathbf{w}$ and $\Sigma$, also on $\bar{\mathbf{w}}_h$. 
\end{lem}
\begin{proof}
Consider the case $\theta=\theta _1 \leq  \theta_2$. We only need to prove  
\begin{align}\label{acc+}
\| \tilde{\mathbf{w}}_h-\mathbf{w}_h\| _{\infty}  \leq C \| {\mathbf{w}}_h-\mathbf{w}\| _{\infty},
\end{align}
from which (\ref{acc}) follows by further using the triangle inequality.    We proceed to prove (\ref{acc+}) in several steps. 

\underline{Step 1.}  From (\ref{bb}), it follows that
\begin{align*}
\| \tilde{\mathbf{w}}_h-\mathbf{w}_h\| _{\infty}=&(1-\theta _1)\|\bar{\mathbf{w}}_h-\mathbf{w}_h\|_{\infty}\\
=&\left(1-\frac{r_0-r(\bar{\mathbf{w}}_h)}{r_{\max}-r(\bar{\mathbf{w}}_h)} \right)\|\bar{\mathbf{w}}_h-\mathbf{w}_h \|_{\infty}\\
=&\frac{\max \limits _{x\in I} |\bar{\mathbf{w}}_h-\mathbf{w}_h(x)|}{r_{\max}-r(\bar{\mathbf{w}}_h)}(r_{\max}-r_0).
\end{align*}

\underline{Step 2.} The overshoot estimate. Since $\mathbf{w}(x)\in \Sigma$, 
\begin{align*}
r_{\max}-r_0 & \leq \max _{x\in I}\left(r(\mathbf{w}_h(x))-r(\mathbf{w}(x))\right)\\
& \leq \|\triangledown r\| _{\infty}\|\mathbf{w}_h-\mathbf{w}\|_{\infty} \leq C_1\|\mathbf{w}_h-\mathbf{w}\|_{\infty},
\end{align*}
where $C_1=\sqrt{1-p'(v^*/2)}$ by Lemma \ref{lem:4}.\\

\underline{Step 3.} Let $\hat{x}^{\alpha}$ ($\alpha=1,\cdots ,k+1$) be  the Gauss quadrature points such that 
\begin{align}\label{gq}
\bar{\mathbf{w}}_h=\sum ^{k+1}_{\alpha =1}\hat{w}_{\alpha}\mathbf{w}_h(\hat{x}^{\alpha}), 
\end{align}
with $\hat{w}_{\alpha}>0$ and $\sum ^{k+1}_{\alpha =1}\hat{w}_{\alpha}=1$. Using these quadrature points, 
we also have 
$$
\mathbf{w}_h(x)-\bar{\mathbf{w}}_h=\sum ^{k+1}_{\alpha =1}
(\mathbf{w}_h(\hat{x}^{\alpha})-\bar{\mathbf{w}}_h)l_{\alpha}(\xi), \quad \xi =(x-a)/(b-a), 
$$
where $l_{\alpha}(\xi)$ ($\alpha=1,\cdots ,N$) are  the Lagrange interpolating polynomials.
Hence, we have
\begin{align*}
\max \limits _{x\in I} |\bar{\mathbf{w}}_h-\mathbf{w}_h(x)|\leq &\max _{\xi \in [0, 1]}\sum \limits^{k+1}_{\alpha =1}|l_{\alpha}(\xi)||\bar{\mathbf{w}}_h-\mathbf{w}_h(\hat{x}^{\alpha})|\\
\leq &C_2\max _{\alpha}|\bar{\mathbf{w}}_h-\mathbf{w}_h(\hat{x}^{\alpha})|, 
\end{align*}  
where $C_2=\Lambda _{k+1}([0, 1])\doteq\max \limits_{\xi \in [0, 1]}\sum \limits^{k+1}_{\alpha =1}|l_{\alpha}(\xi)|$ is the Lebesgue constant.

\underline{Step 4.} Change of variables. Since $g(v)$ is increasing and concave, we have 
\begin{align*}
u_h(x)& =\frac{r(\mathbf{w}_h(x))+s(\mathbf{w}_h(x))}{2}, \quad \quad v_h(x)=g^{-1}\left(\frac{s(\mathbf{w}_h(x))-r(\mathbf{w}_h(x))}{2}\right), \\
\bar{u}_h(x)& =\frac{r(\bar{\mathbf{w}}_h(x))+s(\bar{\mathbf{w}}_h(x))}{2}, \quad \quad \bar{v}_h(x)=g^{-1}\left(\frac{s(\bar{\mathbf{w}}_h(x))-r(\bar{\mathbf{w}}_h(x))}{2}\right).
\end{align*}
If we set 
$$
E=\max \limits _{\alpha}|r(\bar{\mathbf{w}}_h)-r(\mathbf{w}_h(\hat{x}^{\alpha}))|+\max \limits _{\alpha}|s(\bar{\mathbf{w}}_h)-s(\mathbf{w}_h(\hat{x}^{\alpha}))|,
$$
then 
\begin{align*}
\max _{\alpha}|\bar{\mathbf{w}}_h-\mathbf{w}_h(\hat{x}^{\alpha})|
& \leq  \sqrt{ \max \limits _{\alpha}|\bar{u}_h-u_h(\hat{x}^{\alpha})|^2+ \max \limits _{\alpha}|\bar{v}_h-v_h(\hat{x}^{\alpha})|^2}\\ 
& \leq \frac{1}{2}\sqrt{1+ \| (g^{-1})'(\cdot)\|^2_{\infty}} E \\
& \leq C_3 E,
\end{align*}
where $C_3=\frac{1}{2} \sqrt{1-\frac{1}{p'(\| v\|_{\infty}+1)}}$ by Lemma \ref{lem:4}.\\

Step 5. We are now left to show the uniform bound of
\begin{align*}
\frac{E}{r_{\max}-r(\bar{\mathbf{w}}_h)}=\frac{\max \limits _{\alpha}|r(\bar{\mathbf{w}}_h)-r(\mathbf{w}_h(\hat{x}^{\alpha}))|+\max \limits _{\alpha}|s(\bar{\mathbf{w}}_h)-s(\mathbf{w}_h(\hat{x}^{\alpha}))|}{r_{\max}-r(\bar{\mathbf{w}}_h)},
\end{align*}
which is equivalent to show the boundedness of
\begin{align*}
\max \left\lbrace \frac{\hat{r}_{\max}-r(\bar{\mathbf{w}}_h)}{r_{\max}-r(\bar{\mathbf{w}}_h)}, \frac{r(\bar{\mathbf{w}}_h)-\hat{r}_{\min}}{r_{\max}-r(\bar{\mathbf{w}}_h)} \right\rbrace+\max \left\lbrace \frac{s(\bar{\mathbf{w}}_h)-\hat{s}_{\min}}{r_{\max}-r(\bar{\mathbf{w}}_h)},\frac{\hat{s}_{\max}-s(\bar{\mathbf{w}}_h)}{r_{\max}-r(\bar{\mathbf{w}}_h)} \right\rbrace,
\end{align*}
where
\begin{align*}
& \hat{r}_{\max}\doteq \max \limits_{\alpha}r(\mathbf{w}_h(\hat{x}^{\alpha})), \quad \hat{r}_{\min}\doteq \min \limits_{\alpha}r(\mathbf{w}_h(\hat{x}^{\alpha})), \\
&  \hat{s}_{\max}\doteq \max \limits_{\alpha}s(\mathbf{w}_h(\hat{x}^{\alpha})), \quad \hat{s}_{\min}\doteq \min \limits_{\alpha}s(\mathbf{w}_h(\hat{x}^{\alpha})).
\end{align*}
Recall that $\theta= \theta _1\leq \theta _2$, therefore
\begin{align*}
\frac{r_0-r(\bar{\mathbf{w}}_h)}{r_{\max}-r(\bar{\mathbf{w}}_h)}\leq \frac{s(\bar{\mathbf{w}}_h)-s_0}{s(\bar{\mathbf{w}}_h)-s_{\min}} \quad \Rightarrow \quad 
\frac{s(\bar{\mathbf{w}}_h)-s_{\min}}{r_{\max}-r(\bar{\mathbf{w}}_h)}\leq \frac{s(\bar{\mathbf{w}}_h)-s_0}{r_0-r(\bar{\mathbf{w}_h})}.
\end{align*}
Then
\begin{align*}
&\max \left\lbrace\frac{\hat{r}_{\max}-r(\bar{\mathbf{w}}_h)}{r_{\max}-r(\bar{\mathbf{w}}_h)}, \frac{r(\bar{\mathbf{w}}_h)-\hat{r}_{\min}}{r_{\max}-r(\bar{\mathbf{w}}_h)} \right\rbrace +\max \left\lbrace\frac{s(\bar{\mathbf{w}}_h)-\hat{s}_{\min}}{r_{\max}-r(\bar{\mathbf{w}}_h)},\frac{\hat{s}_{\max}-s(\bar{\mathbf{w}}_h)}{r_{\max}-r(\bar{\mathbf{w}}_h)} \right\rbrace \\
\leq &C_4\left(\max \left\lbrace 1, \frac{r(\bar{\mathbf{w}}_h)-\hat{r}_{\min}}{r_{\max}-r(\bar{\mathbf{w}}_h)} \right\rbrace+\max \left\lbrace 1,\frac{\hat{s}_{\max}-s(\bar{\mathbf{w}}_h)}{s(\bar{\mathbf{w}}_h)-s_{\min}}\right\rbrace \right),
\end{align*}
where $C_4=2\max \{1,\frac{s(\bar{\mathbf{w}}_h)-s_0}{r_0-r(\bar{\mathbf{w}}_h)}\}$. Note that for the case ${\theta}=\theta _2$, $C_4$ needs to be replaced by $2\max\{1,\frac{r_0-r(\bar{\mathbf{w}}_h)}{s(\bar{\mathbf{w}}_h)-s_0}\}$.\\

Step 6. Note that (\ref{gq}) when combined with convexity of $r$ yields 
\begin{align*}
&r(\bar{\mathbf{w}}_h)=r\left(\sum \limits^{k+1}_{\alpha =1}\hat{w}_{\alpha}\mathbf{w}_h(\hat{x}^{\alpha})\right)\leq \sum \limits ^{k+1}_{\alpha =1}\hat{w}_{\alpha}r(\mathbf{w}_h(\hat{x}^{\alpha})).
\end{align*}
Assume $\hat{r}_{\min}$ is achieved at $\hat{x}^1$, then  
\begin{align*}
&\hat{w}_1(r(\bar{\mathbf{w}}_h)-\hat{r}_{\min})\leq \sum \limits^{k+1}_{\alpha =2}\hat{w}_{\alpha}(r(\mathbf{w}_h(\hat{x}^{\alpha}))-r(\bar{\mathbf{w}}_h))\leq (r_{\max}-r(\bar{\mathbf{w}}_h))\sum \limits ^{k+1}_{\alpha =2}\hat{w}_{\alpha}\\
\Rightarrow \quad &\frac{r(\bar{\mathbf{w}}_h)-\hat{r}_{\min}}{r_{\max}-r(\bar{\mathbf{w}}_h)}\leq \frac{1-\min \limits _{\alpha}\hat{w}_{\alpha}}{\min \limits_{\alpha}\hat{w}_{\alpha}}.
\end{align*}
Similarly, we can show that $\frac{\hat{s}_{\max}-s(\bar{\mathbf{w}}_h)}{s(\bar{\mathbf{w}}_h)-s_{\min}}\leq  \frac{1-\min \limits _{\alpha}\hat{w}_{\alpha}}{\min \limits_{\alpha}\hat{w}_{\alpha}}$, due to the concavity of $s$. 

The above steps together have verified the claimed estimate (\ref{acc+}),  with the bounding constant $C=\prod ^5_{i=1}C_i$, where 
\begin{equation}\label{c4}
C_4=2\max \left\{1,\frac{s(\bar{\mathbf{w}}_h)-s_0}{r_0-r(\bar{\mathbf{w}}_h)}, \frac{r_0-r(\bar{\mathbf{w}}_h)}{s(\bar{\mathbf{w}}_h)-s_0} \right\}, \quad C_5\doteq 2\max\left\{1, \frac{1-\min \limits _{\alpha}\hat{w}_{\alpha}}{\min \limits_{\alpha}\hat{w}_{\alpha}}\right\}.
\end{equation}
\end{proof}

\begin{rem}
It is noticed that when $\bar{\mathbf{w}}_h$ is close enough to the boundary of $\Sigma$, the constant $C_4$ can become quite large, indicating the possibility of accuracy deterioration in some cases. In appendix \ref{appsec:C4}, we present two examples to show that the magnitude of $C_4$ can be uniformly bounded or quite large.
\end{rem}
\begin{rem} \red{Our proof is the first attempt at rigorous justification of the accuracy-preservation by the IRP limiter for system case.}  Related techniques to Step 3 and Step 6 are also found in the proof of Lemma 7 \cite{Zh17} for scalar conservation laws, where the proof was accredited to Mark Ainsworth. An alternative proof for the accuracy of the limiter is given in \cite{LYH14}, in which the authors consider some weighted averages for scalar Fokker-Planck equations.  
\end{rem}

In summary, we have the following result.

\begin{thm}
\label{BigThm}
Let $\mathbf{w}_h(x)=(v_h(x),u_h(x))^T$ be a polynomial approximation to the smooth function $\mathbf{w}(x)=(v(x),u(x))^T\in \Sigma$, over bounded cell $I$, and $\bar{\mathbf{w}}_h \in \Sigma _0$. Define $r_{\text{max}} = \max \limits_{x\in I } r(\mathbf{w}_h(x))$, $s_{\text{min}} =\min \limits_{x\in I}s(\mathbf{w}_h(x))$. Then the modified polynomial
\begin{align*}
\tilde{\mathbf{w}}_h(x)={\theta}\mathbf{w}_h(x)+(1-{\theta})\bar{\mathbf{w}}_h, \quad {\theta}=\min \{1, \theta _1,\theta _2\},
\end{align*}
where 
\begin{align*}
\theta _1=\frac{r_0-r(\bar{\mathbf{w}}_h)}{r_{\text{max}}-r(\bar{\mathbf{w}}_h)}, \quad \quad \theta _2=\frac{s(\bar{\mathbf{w}}_h)-s_0}{s(\bar{\mathbf{w}}_h)-s_{\text{min}}},
\end{align*}
satisfies the following three properties:
\begin{enumerate}
\item the cell average is preserved, i.e. $\bar{\mathbf{w}}_h=\bar{\tilde{\mathbf{w}}}_h$;
\item it entirely lies within invariant region $\Sigma$, i.e. $r(\tilde{\mathbf{w}}_h)\leq r_0$, $s(\tilde{\mathbf{w}}_h)\geq s_0$;
\item high order of accuracy is maintained, i.e. $\| \tilde{\mathbf{w}}_h-\mathbf{w}\| _{\infty} \leq C \| {\mathbf{w}}_h-\mathbf{w}\| _{\infty}$,
where $C$ is a positive constant that only depends on $\bar{\mathbf{w}}_h, \mathbf{w}$, and the invariant region $\Sigma$.
\end{enumerate}
\end{thm}
\begin{rem} \red{The IRP limiter (\ref{bb})--(\ref{aa}) remains valid for $\Sigma_1$ in (\ref{s1}) and $\Sigma_2$ in (\ref{s2}).  The analysis above and the accuracy results in Theorem \ref{BigThm} can be easily generalized to both the one-dimensional isentropic gas dynamic system in Eulerian coordinates (\ref{IGDE}) and the shallow water equations (\ref{IGDE+}). }  
\end{rem}

To show when this simple limiter can be incorporated into an existing high order numerical scheme,  we present the following IRP algorithm.  
\subsection{Algorithm}\label{al}  
Let $\mathbf{w}^n_h$ be the numerical solution generated from a high order finite-volume-type scheme of an abstract form
\begin{equation}\label{abstsche}
\mathbf{w}^{n+1}_h=\mathcal{L}(\mathbf{w}^n_h),
\end{equation}
where $\mathbf{w}^n_h=\mathbf{w}^n_h(x)\in V_h$, and $V_h$ is a finite element space of piecewise polynomials of degree $k$ in each computational cell, i.e.,  
\begin{align}\label{Vh}
V_{h}=\{\upsilon:~\upsilon |_{ I_{j}}\in
 \mathbb{P}^{k}(I_{j})\}.
\end{align}
Assume $\lambda =\frac{\Delta t}{h}$ is the mesh ratio, where $h=\max \limits_j |I_j|$.
Provided that scheme (\ref{abstsche}) has the following property: there exists $\lambda _0$, and a test set $S$ such that if
\begin{equation}\label{suff}
\lambda \leq \lambda _0 \quad {\rm and} \quad  \mathbf{w}^n_h(x)\in {\Sigma} \quad {\rm for}\quad  x\in S
\end{equation}
then
\begin{equation}\label{prop}
\bar{\mathbf{w}}^{n+1}_h\in \Sigma _0; 
\end{equation}
 the limiter can then be applied with $I$ replaced by $S_I:=S\cap I_j$ in (\ref{aa}), i.e., 
\begin{align}
\label{defqp}
r_{\text{max}} = \max_{x\in S_I } r(\mathbf{w}_h(x)), \quad s_{\text{min}} =\min_{x\in S_I}s(\mathbf{w}_h(x)).
\end{align}
The algorithm can be described as follows: 

\textbf{Step 1.} Initialization: take the piecewise $L^2$ projection of $\mathbf{w}_0$ onto $V_h$, such that
\begin{equation}\label{L2Projection}
\int_{I_j} (\mathbf{w}^0_h(x)-\mathbf{w}_0(x))\phi(x)dx =0, \quad \forall \phi \in V_h.
\end{equation}
Also from $\mathbf{w}_0$, we compute $r_0$ and $s_0$ as defined in (\ref{r0s0}) to determine the invariant region $\Sigma$.

\textbf{Step 2.} Impose the modified limiter (\ref{bb}), (\ref{limiter}) with (\ref{defqp}) on $\mathbf{w}^n_h$  for $n=0, 1, \cdots $ to obtain $\tilde{\mathbf{w}}^n_h$.

\textbf{Step 3.} Update by the scheme:
\begin{equation}
\mathbf{w}^{n+1}_h=\mathcal{L}(\tilde{\mathbf{w}}^n_h).
\end{equation}
Return to \textbf{Step 2}.

\begin{rem}With the modification defined in (\ref{defqp}), both (1) and (3) in Theorem \ref{BigThm} remain valid, and the modified polynomials 
 lie within the invariant region $\Sigma$ only for $x\in S_I$.  Moreover,  the limiter (\ref{bb}), (\ref{limiter}) with (\ref{defqp}) can enhance the efficiency of computation, and we shall use this modified IRP limiter in our numerical experiments. 
\end{rem}


\section{IRP high order schemes}
In this section we discuss the IRP property of high order schemes for both the p-system (\ref{ps}) and  the viscous p-system (\ref{pes}). These systems are known to share the same invariant region $\Sigma$ (see \cite{Sm94}).
For simplicity,  we will assume periodic boundary conditions from now on. 
\subsection{IRP schemes for (\ref{ps})}
Rewrite the p-system into a compact form
\begin{equation}\label{cptfm}
\mathbf{w}_t+F(\mathbf{w})_x=0,
\end{equation}
with $\mathbf{w}=(v,u)^\top$ and $F(\mathbf{w})=(-u,p(v))^\top$.   We begin with the first order finite volume scheme 
\begin{align}\label{1stFV}
\mathbf{w}^{n+1}_j=\mathbf{w}^n_j-\lambda\left(\hat{F}(\mathbf{w}^n_j,\mathbf{w}^n_{j+1})-\hat{F}(\mathbf{w}^n_{j-1},\mathbf{w}^n_j)\right),
\end{align}
with the Lax-Friedrich flux defined by 
\begin{align}\label{LF-}
\hat{F}(\mathbf{w}^n_{i},\mathbf{w}^n_{i+1})=\frac{1}{2}\left(F(\mathbf{w}_i^n)+F(\mathbf{w}^n_{i+1})-\sigma (\mathbf{w}^n_{i+1} - \mathbf{w}^n_i) \right), \quad i=j-1, j, 
\end{align} 
where $\mathbf{w}^{n}_j$ is approximation to the average of $\mathbf{w}$ on $I_j=[x_{j-1/2}, x_{j+1/2}]$ at time level $n$, and $\sigma$ is a positive number dependent upon $\mathbf{w}_{l}^n$ for $l=j-1,j,j+1$.  Given $\mathbf{w}^n_j\in \Sigma$, we want to find a proper $\sigma $ and a CFL condition so that $\mathbf{w}^{n+1}_j\in \Sigma$. 

We can rewrite (\ref{1stFV}) as
\begin{align}\label{LFcomb}
\mathbf{w}^{n+1}_j=\left(1-\lambda\sigma \right)\mathbf{w}^n_j+\lambda\sigma {\mathbf{w}}^*,
\end{align}
where
\begin{align*}
{\mathbf{w}}^*= \left(\frac{\mathbf{w}^n_{j-1}+\mathbf{w}^n_{j+1}}{2}-\frac{F(\mathbf{w}^n_{j+1})-F(\mathbf{w}^n_{j-1})}{2\sigma} \right).
\end{align*}
It is known, see \cite[Section 14.1]{LeVequeCL} for example, that for {
$\sigma\geq \max \limits_{v^n_j,v^n_{j\pm 1}} \sqrt{-p'(v)}$}, 
\begin{align*}
{\mathbf{w}}^*=\frac{1}{2\sigma}\int ^{\sigma}_{-\sigma}\mathbf{w}(t\xi, t)d\xi,
\end{align*}
where $\mathbf{w}(x, t)$
is the exact Riemann solution to (\ref{cptfm}) subject to initial data
\begin{align*}
\mathbf{w}(x ,0)=
\begin{cases}
\mathbf{w}^n_{j-1}, \quad x <0, \\
\mathbf{w}^n_{j+1}, \quad x >0.
\end{cases}
\end{align*}
For $\mathbf{w}^n_{j\pm1} \in \Sigma$, we have $\mathbf{w}(x, t) \in \Sigma$ since $\Sigma$ is the invariant region of the p-system. Therefore ${\mathbf{w}}^*$ lies in $\Sigma _0$ by Lemma \ref{lem:2}.  Since $\mathbf{w}^{n+1}_j$ is a convex combination of two vectors $\mathbf{w} \in \Sigma $ and ${\mathbf{w}}^* \in \Sigma_0 $ for $\lambda \sigma  \leq 1$, we then have $\mathbf{w}^{n+1}_j \in \Sigma_0$.  In summary, we have the following.
\begin{thm}[First order scheme]\label{thm:1stFV}
The first order scheme (\ref{1stFV}), (\ref{LF-}) with {$\sigma \geq \max \limits_{v^n_j,v^n_{j\pm 1}} \sqrt{-p'(v)}$}, preserves the invariant region $\Sigma$, i.e.,
$$
\text{if}\quad  \mathbf{w}^n_i\in \Sigma \quad \text{for}\quad  i=j, j\pm1, \quad  \text{then} \quad \mathbf{w}^{n+1}_j\in \Sigma_0,
 $$ 
under the CFL condition 
\begin{align}\label{1stfvcfl}
\lambda \sigma \leq 1 \quad \text{with} \quad \sigma \geq \max \limits_{v^n_j,v^n_{j\pm1}} \sqrt{-p'(v)}.
\end{align}
\end{thm}
\begin{rem} Note that for scheme (\ref{1stFV}), the flux parameter $\sigma$ at each $x_{j+1/2}$ needs to be dependent on $\mathbf{w}_i^n$ with $i=j-1, j, j+1, j+2$, instead of the usual local Lax-Friedrichs flux with $\sigma$ depending on  $\mathbf{w}_i^n$ with $i=j, j+1$. An alternative proof of the positivity-preserving property was given in the appendix in \cite{PerthameShu1996} for compressible Euler equations; their technique when applied to (\ref{1stFV}) with the global Lax-Friedrich flux leads to $\lambda \sigma\leq \frac{1}{2}$. In contrast, the CFL condition (\ref{1stfvcfl}) is more relaxed.  We also note that using the quadratic structure of the flux function $F$ in the full compressible Euler equation,  Zhang was able to obtain a relaxed CFL in \cite{Zh17} with the local Lax-Friedrich flux by direct verification without reference to the Riemann solution.   
\end{rem}

We next consider a (k+1)th-order scheme with reconstructed polynomials or approximation polynomials of degree $k$.  With forward Euler time discretization, the cell average evolves by  
\begin{align}\label{eq: High order FV scheme}
\bar{\mathbf{w}}^{n+1}_j=\bar{\mathbf{w}}^n_j-\lambda [\hat{F}(\mathbf{w}^{-}_{j+\frac{1}{2}},\mathbf{w}^{+}_{j+\frac{1}{2}})-\hat{F}(\mathbf{w}^{-}_{j-\frac{1}{2}},\mathbf{w}^{+}_{j-\frac{1}{2}})],
\end{align}
where $\bar{\mathbf{w}}^n_j$ is the cell average of $ \mathbf{w}^{n}_h$ on $I_j$ at time level $n$, $\mathbf{w}^{\pm }_{j+\frac{1}{2}}$ are approximations to the point value of $\mathbf{w}$ at $x_{j+1/2}$ and time level $n$ from the left and from the right respectively. The local Lax-Friedrichs flux is taken with 
\begin{align}\label{si}
\sigma \geq \max \limits_{v^-_{j\pm 1/2}, v^+_{j\pm1/2}} \sqrt{-p'(v)}.
\end{align}
We consider an $N-$point 
Legendre Gauss-Lobatto quadrature rule on $I_j$,  with quadrature weights $\hat \omega_i$ on $[-\frac{1}{2},\frac{1}{2}]$ such that  $\sum_{i=1}^N \hat \omega_i =1$, which is exact for integrals of polynomials of degree up to $k$, if $2N-3 \geq k$. 
Denote these quadrature points  on $I_j$ as 
$$
S_j^C:=\{\hat x^i_j, 1\leq i \leq N\},  
$$
where $\hat x^1_j=x_{j-1/2}$ and $\hat x^N_j=x_{j+1/2}$.  The cell average decomposition then takes the form 
\begin{align}\label{avgdec}
\bar{\mathbf{w}}^n_j=\sum_{i=2}^{N-1} \hat\omega_\alpha \mathbf{w}_h^n(\hat x_j^i) +\hat \omega_1
\mathbf{w}^{+}_{j-\frac{1}{2}} +\hat \omega_N \mathbf{w}^{-}_{j+\frac{1}{2}},
\end{align}
where it is known that $\hat \omega_1=\hat \omega_N=1/(N(N-1))$.  Hence  (\ref{eq: High order FV scheme}) can be rewritten as a linear convex combination of the form 
\begin{align}\label{decschm}
\bar{\mathbf{w}}^{n+1}_j=\sum ^{N-1}_{i=2}\hat{\omega}_{i}\mathbf{w}_h^n(\hat{x}^i_j)+\hat{\omega}_1K_1 +\hat{\omega}_NK_N,
\end{align}
where 
\begin{align*}
K_1=\mathbf{w}^+_{j-\frac{1}{2}}-\frac{\lambda}{\hat{\omega}_1}\left(\hat{F}(\mathbf{w}^+_{j-\frac{1}{2}},\mathbf{w}^-_{j+\frac{1}{2}})-\hat{F}(\mathbf{w}^-_{j-\frac{1}{2}},\mathbf{w}^+_{j-\frac{1}{2}}) \right),\\
K_N=\mathbf{w}^-_{j+\frac{1}{2}}-\frac{\lambda}{\hat{\omega}_N}\left(\hat{F}
(\mathbf{w}^-_{j+\frac{1}{2}},\mathbf{w}^+_{j+\frac{1}{2}})- \hat{F}(\mathbf{w}^+_{j-\frac{1}{2}},\mathbf{w}^-_{j+\frac{1}{2}})\right)
\end{align*}
are of the same type as the first order scheme (\ref{1stFV}).  Such decomposition of (\ref{eq: High order FV scheme}),  first introduced by Zhang and Shu (\cite{ZhangShu2010b}) for the compressible Euler equation, suffices for us to conclude the following result. 
\begin{thm}[High order scheme]\label{thm:highFV} 
A sufficient condition for $ \bar{\mathbf{w}}^{n+1}_j\in \Sigma _0$ by scheme (\ref{eq: High order FV scheme}) with (\ref{si}) is 
$$
\mathbf{w}^n_h(x)\in \Sigma \quad \text{for} \; x\in  S_j^C
$$
under the CFL condition 
\begin{align}\label{pscfl}
\lambda \sigma  \leq \frac{1}{N(N-1)} \quad \text{with} \quad \sigma \geq \max _{v^-_{j\pm1/2}, v^{+}_{j\pm1/2}}\sqrt{-p'(v)}, \quad N=\lceil \frac{k+3}{2}\rceil.
\end{align}
\end{thm}

\begin{rem}\label{altcv}
For third order schemes, an alternative decomposition of the cell average can be found through Lagrangian interpolation on the test set 
$S^C_j=x_j+\frac{\Delta x}{2}\{-1,\gamma , 1\}$:
\begin{align}\label{scv}
\bar{\mathbf{w}}^n_j=\frac{1+3\gamma}{6(1+\gamma)}\mathbf{w}^+_{j-\frac{1}{2}}+\frac{2}{3(1-\gamma ^2)}\mathbf{w}^n_h\left(x_j+\frac{\Delta x}{2}\gamma\right)+\frac{1-3\gamma}{6(1-\gamma)}\mathbf{w}^-_{j+\frac{1}{2}},
\end{align}
where $|\gamma|\leq \frac{1}{3}$ to ensure non-negative weights. Theorem \ref{thm:highFV} still holds under this choice of test set.
\end{rem}
\begin{rem}\label{ext}
\red{Even though we proved both Theorem \ref{thm:1stFV} and Theorem  \ref{thm:highFV} only for the p-system (\ref{ps}), the results obtained actually hold for all one-dimensional hyperbolic conservation laws (\ref{cptfm+}) equipped with a convex invariant region, as long as the flux parameter 
$\sigma$ is chosen so that 
$$
\sigma  \geq  \lambda_{\rm max},
$$
where $\lambda_{\rm max}$ is the maximum wave speed (largest eigenvalues of the Jacobian matrix of $F(\mathbf{w})$). A direct calculation shows that  $\lambda_{\rm max}= |u|+ \sqrt{p'(\rho)}$ for (\ref{IGDE}), and $\lambda_{\rm max}= |u|+ \sqrt{gh}$ for (\ref{IGDE+}).
}
\end{rem}

\subsection{IRP DG schemes for (\ref{pes})} 
The p-system with artificial diffusion can be rewritten as
\begin{align}\label{cptfmdf}
\mathbf{w}_t+F(\mathbf{w})_x=\epsilon \mathbf{w}_{xx}
\end{align}
with $\mathbf{w}=(v,u)^\top$ and $F(\mathbf{w})=(-u,p(v))^\top$.  
 The DG  scheme for  (\ref{cptfmdf}) is to find $\mathbf{w}_h\in V_h$ such that
\begin{align}\label{dgv}
\int _{I_j}(\mathbf{w}_h)_t\phi dx=&\int _{I_j}F(\mathbf{w}_h)\phi _xdx-\hat{F}(\mathbf{w}^-_h,\mathbf{w}^+_h)\phi \big |_{\partial I_j}\\ \notag 
&-\int _{I_j}\epsilon (\mathbf{w}_h)_x\phi _xdx+\epsilon \left(\widehat{(\mathbf{w}_h)_x}\phi+(\mathbf{w}_h-\{\mathbf{w}_h\})\phi _x \right)\big |_{\partial I_j}\quad \forall \phi \in V_h,
\end{align}
where $\hat{F}$ is the local Lax-Friedrich flux for the convection part as defined in (\ref{LF-}), see \cite{CLS89};  while the diffusive numerical flux $\widehat{(\mathbf{w}_h)_x}$ is chosen, following \cite{LY10, LY09}, as
\begin{align}\label{df}
\widehat{(\mathbf{w}_h)_x}=\beta _0\frac{[\mathbf{w}_h]}{\Delta x}+\{(\mathbf{w}_h)_x\}+\beta _1\Delta x[(\mathbf{w}_h)_{xx}], 
\end{align}
where $[\cdot]$ denotes the jump of the function and $\{\cdot \}$ denotes the average of the function across the cell interface. The flux parameters $\beta _0$ and $\beta _1$ are  to be chosen from an appropriate range in order to achieve the desired IRP property.  

With forward Euler time discretization of (\ref{dgv}),  the cell average evolves by  
\begin{align}\label{ddgscm}
\bar{\mathbf{w}}^{n+1}_j = \bar{\mathbf{w}}^n_j-\lambda\hat{F}(\mathbf{w}^-_h,\mathbf{w}^+_h)\big |_{\partial I_j}+\epsilon \mu \Delta x\widehat{(\mathbf{w}_h)_x}\big |_{\partial I_j},
\end{align}
where  the mesh ratios $\mu =\frac{\Delta t}{\Delta x^2}$ and  $\lambda= \frac{\Delta t}{\Delta x}$.
 
We want to find sufficient conditions on $\lambda, \mu$ and a test set S such that if $\mathbf{w}^{n}_h(x)\in \Sigma$ on $S$, then $\bar{\mathbf{w}}^{n+1}_j \in \Sigma_0$ for all $j$. 
To do so,  we rewrite (\ref{ddgscm}) as  
\begin{align}\label{split}
\bar{\mathbf{w}}^{n+1}_j = \frac{1}{2}{C}_j+\frac{1}{2}{D}_j,
\end{align}
where
\begin{align*}
{C}_j = & \bar{\mathbf{w}}^n_j - 2\lambda\left(\hat{F}(\mathbf{w}^-_{j+\frac{1}{2}},\mathbf{w}^+_{j+\frac{1}{2}})-\hat{F}(\mathbf{w}^-_{j-\frac{1}{2}},\mathbf{w}^+_{j-\frac{1}{2}})  \right),\\
{D}_j = & \bar{\mathbf{w}}^n_j+2\epsilon \mu \left(\Delta x\widehat{(\mathbf{w}^n_h)_x}\big |_{x_{j+\frac{1}{2}}}-\Delta x\widehat{(\mathbf{w}^n_h)_x}\big |_{x_{j-\frac{1}{2}}} \right).
\end{align*}
Recall the result obtained in Theorem \ref{thm:highFV}, we see that for any $k>0$, the sufficient condition for  
$
 {C}_j \in \Sigma_0
 $
 is the following:
 $$
\mathbf{w}^n_h(x)\in \Sigma \quad \text{for} \; x\in  S_j^C
$$
and the CFL condition 
\begin{align}\label{pscflD}
\lambda \sigma  \leq \frac{1}{2N(N-1)} \quad \text{with} \quad \sigma \geq \max _{v^-_{j\pm1/2}, v^{+}_{j\pm1/2}}\sqrt{-p'(v)},
\end{align}
where the $N\geq (3+k)/2$ is required to hold. The remaining task is to find sufficient conditions so that $ {D}_j \in \Sigma_0$. These together when combined with (\ref{split}) will give $\bar{\mathbf{w}}^{n+1}_j \in \Sigma_0$ as desired. 

In order to ensure that ${D}_j\in \Sigma _0$, we follow Liu and Yu in \cite{LYH14} so as to obtain some sufficient conditions for $k=1, 2$, respectively.  The results are summarized in the following two lemmas.
\begin{lem}\label{lem:p1} 
Consider scheme (\ref{ddgscm}) with $k=1$ and $\beta_0 \geq 1/2$. The sufficient condition for  ${D}_j \in \Sigma_0$ is 
$$
\mathbf{w}^n_h(x)\in \Sigma \quad \text{for} \; x\in \mathop \cup \limits_{i=j-1}^{j+1} S_i^D,
$$
where 
\begin{align*}
S^D_i=x_i+\frac{\Delta x}{2}\{-\gamma, \gamma\}
\end{align*}
for $\gamma  \in[-1, 1] \setminus \{0\}$ with
\begin{align}\label{p1gamma}
\left|1-\frac{1}{\beta _0} \right|\leq |\gamma| \leq 1,
\end{align}
under the condtion  $\mu\leq \frac{1}{4\epsilon \beta _0}$.
\end{lem}
\begin{lem}\label{lem:p2} 
Consider scheme (\ref{ddgscm}) with $k=2$ and 
\begin{align}\label{p2beta}
\beta _0\geq 1, \quad \frac{1}{8} \leq  \beta _1 \leq  \frac{1}{4}.
\end{align}
The sufficient condition for  ${D}_j \in \Sigma_0$ is 
$$
\mathbf{w}^n_h(x)\in \Sigma \quad \text{for} \; x\in \mathop \cup \limits_{i=j-1}^{j+1} S_i^D,
$$
where 
\begin{align*}
S^D_i=x_i+\frac{\Delta x}{2}\{-1, \gamma , 1\}
\end{align*}
for $\gamma  \in[-1, 1]$  satisfying 
\begin{align}\label{p2gamma}
|\gamma |< \frac{1}{3}  \text{ and } |\gamma |\leq 8\beta _1-1
\end{align}
under the CFL conditions $\mu\leq \mu _0$, where 
\begin{align}\label{mu0}
\mu _0=\frac{1}{12\epsilon} \min \left\{\frac{1\pm 3\gamma}{(1\pm \gamma)\beta _0-\theta}, \frac{2}{\theta}\right\}, \quad 
\theta:=2- 8\beta _1 \in [0, 1]. 
\end{align}
\end{lem}

For reader's convenience,  we outline the related calculations  leading to the above results in appendix \ref{appsec:lems}.
\begin{rem}\label{ddg} 
\red{
In \cite{LYH14}, Liu and Yu proposed a third order maximum-principle-preserving DDG 
scheme over Cartesian meshes for the linear Fokker-Planck equation 
\begin{align}\label{fp}
u_t =\nabla_x \cdot (\nabla_x u +\nabla_x V(x) u),
\end{align} 
where the potential $V$ is given, provided the flux parameter $(\beta_0, \beta_1)$ falls into the range $\beta_0 \geq 1$ and $\beta_1\in [1/8, 1/4]$.
The maximum-principle for (\ref{fp}) means that if $u_0\in [c_1, c_2]e^{-V(x)}$, then $u(x, t) \in [c_1, c_2]e^{-V(x)}$ for all $t>0$, which implies the usual maximum-principle for diffusion $u(x, t)\in [c_1, c_2]$ for all $t>0$.  Extension to unstructured meshes is non-trivial, we refer to 
 \cite{CHY16} for some recent results in solving diffusion equations over triangular meshes. }   
\end{rem}

Combining (\ref{pscflD}), with Lemma \ref{lem:p1} and Lemma \ref{lem:p2}, we are able to establish the following result.
\begin{thm}\label{thm: viscosity}
A sufficient condition for $\bar{\mathbf{w}}^{n+1}_j\in \Sigma _0$ in scheme  (\ref{ddgscm}) is
\begin{align*}
\mathbf{w}^n_h(x)\in \Sigma \quad \text{for }x\in \mathop \cup \limits_{i=j-1}^{j+1} S_i,
\end{align*}
under the CFL conditions $\lambda \leq \lambda_0$ and $\mu\leq \mu_0$,  where 
\begin{itemize}
\item[(i)] for second order scheme with $\beta _0\geq \frac{1}{2}$,  
\begin{align*}
S_i=x_i+\frac{\Delta x}{2}\{-1,1\}, \quad \lambda_0 =\frac{1}{4\sigma},\quad \mu_0 =\frac{1}{4\epsilon\beta _0};
\end{align*}
\item[(ii)] for third order scheme with $\beta _0\geq 1$ and $\frac{1}{8}\leq \beta _1\leq \frac{1}{4}$, 
\begin{align*}
&S_i=x_i+\frac{\Delta x}{2}\{-1,\gamma,1\},\\
&\lambda_0 =\frac{1}{12\sigma}, \quad \mu_0= \frac{1}{12\epsilon} \min\left\{\frac{1\pm 3\gamma}{(1\pm \gamma)\beta _0-\theta}, \frac{2}{\theta}\right\}, \quad \theta:=2- 8\beta _1 \in [0, 1],
\end{align*}
where $|\gamma|<\frac{1}{3}$.
\end{itemize}
\end{thm}
\begin{proof} (i)  For the second order scheme with $\beta _0\geq \frac{1}{2}$, we have $S^C_j=x_j+\frac{\Delta x}{2}\{-1,1\}$ and $S^D_j=x_j+\frac{\Delta x}{2}\{-\gamma, \gamma\}$, where $\gamma$ satisfies (\ref{p1gamma}). Therefore, we take $\gamma=1$ so that  $S_j=S^C_j\cup S^D_j=S^C_j$. \\
(ii) {For the third order scheme with $\beta _0\geq 1$ and $\frac{1}{8}\leq \beta _1\leq \frac{1}{4}$, $S^D_j=x_j+\frac{\Delta x}{2}\{-1, \gamma, 1\}$, where $\gamma$ satisfies (\ref{p2gamma}). While for convection part, we consider the alternative decomposition of the cell average (\ref{scv}) in Remark \ref{altcv}, so that $S_j=S^C_j=S^D_j$.}


The CFL conditions are obtained directly by using  (\ref{pscfl}) and the results stated in two Lemmas above.
\end{proof}

\begin{rem}\label{SSP}
It is known that the high order Strong Stability Preserving (SSP) time discretizations are convex combinations of forward Euler, therefore Theorems \ref{thm:highFV}  and  \ref{thm: viscosity} remain valid when using high order SSP schemes.  
In our numerical experiments we will adopt such high order SSP time discretization so as to also achieve high order accuracy in time.    
\end{rem}

\section{Numerical examples}
In this section, we present numerical examples for testing the IRP limiter and the performance of the IRP DG scheme (\ref{dgv}) with the local Lax-Friedrich flux (\ref{LF-}), using the third order SSP Runge-Kutta (RK3) method for time discretization. $\gamma =1.4$ is taken in all of the examples.

The semi-discrete DG scheme is a closed ODE system 
$$
\frac{d}{dt}\mathbf{W}=L(\mathbf{W}),
$$
where $\mathbf{W}$ consists of the unknown coefficients of the spatial basis, and  $L$ is the corresponding spatial operator. The third order SSP RK3 in \cite{GST2001, ShuOsher1988} reads as 
\begin{equation}\label{RK3}
\begin{aligned}
\mathbf{W}^{(1)}=&\mathbf{W}^n+\Delta tL(\mathbf{W}^n),\\
\mathbf{W}^{(2)}=&\frac{3}{4}\mathbf{W}^n+\frac{1}{4}\mathbf{W}^{(1)}+\frac{1}{4}\Delta tL(\mathbf{W}^{(1)}),\\
\mathbf{W}^{n+1}=&\frac{1}{3}\mathbf{W}^n+\frac{2}{3}\mathbf{W}^{(2)}+\frac{2}{3}\Delta tL(\mathbf{W}^{(2)}).
\end{aligned}
\end{equation}
We apply the limiter at each time stage in the RK3 method. Notice that (\ref{RK3}) is a linear convex combination of Euler Forward method, therefore the invariant region is preserved by the full scheme if it is preserved at each time stage.


In the first four examples, the initial condition is chosen as
\begin{align*}
v_0(x)=2-\sin (x), \quad u_0(x)=1, \quad x\in [0,2\pi],
\end{align*}
and the boundary condition is periodic. Using (\ref{r0s0}) we obtain  $r_0=1$ and $s_0=1$, so that the invariant region 
$\Sigma=\{(v, u)\big |\quad r(v, u)\leq 1, \; s(v, u) \geq 1\}$. We fix the initial mesh $h=\frac{2\pi}{32}$ for all four examples. \\

\noindent{\bf Example 1.} {\sl Accuracy test of the limiter}\\
We test the performance of the IRP limiter introduced in Section \ref{sec:3} by comparing the accuracy of the $L^2$ projection (\ref{L2Projection}) with and without limiter (\ref{bb}). From Tables \ref{tab:1P1} and \ref{tab:1P2}, we see that the IRP limiter preserves the accuracy of high order approximations well.\\

\begin{table}[htbp]
\centering
\begin{tabular}{|c|cccc||cccc|}
\hline
$P^1$&\multicolumn{4}{|c||}{Projection without limiter}  &\multicolumn{4}{|c|}{Projection with limiter}
\\ \hline
$\Delta x$ &$L^{\infty}$ error &Order &$L^1$ error & Order &$L^{\infty}$ error &Order &$L^1$ error & Order\\
\hline
h &1.34E-04 &/    &8.53E-05 &/    &6.70E-04 &/    &1.11E-04 &/\\
h/2  &3.35E-05 &2.00 &2.13E-05 &2.00 &1.67E-04 &2.00 &2.45E-05 &2.18\\
h/4 &8.37E-06 &2.00 &5.33E-06 &2.00 &4.18E-05 &2.00 &5.72E-06 &2.10\\
h/8 &2.09E-06 &2.00 &1.33E-06 &2.00 &1.05E-05 &2.00 &1.38E-06 &2.05\\
h/16 &5.23E-07 &2.00 &3.33E-07 &2.00 &2.62E-06 &2.00 &3.39E-07 &2.02\\
\hline
\end{tabular}
\vspace{5pt}
\caption{Accuracy of the $L^2$ projection using $P^1$ approximation.}
\label{tab:1P1}
\end{table}

\begin{table}[htbp]
\centering
\begin{tabular}{|c|cccc||cccc|}
\hline
$P^2$&\multicolumn{4}{|c||}{Projection without limiter}  &\multicolumn{4}{|c|}{Projection with limiter}
\\ \hline
$\Delta x$ &$L^{\infty}$ error &Order &$L^1$ error & Order &$L^{\infty}$ error &Order &$L^1$ error & Order\\
\hline
h  &9.16E-06 &/    &5.86E-06 &/    &9.16E-06 &/    &5.95E-06 &/\\
h/2  &1.15E-06 &3.00 &7.32E-07 &3.00 &1.15E-06 &3.00 &7.35E-07 &3.02\\
h/4 &1.44E-07 &3.00 &9.15E-08 &3.00 &1.44E-07 &3.00 &9.16E-08 &3.00\\
h/8 &1.80E-08 &3.00 &1.14E-08 &3.00 &1.80E-08 &3.00 &1.14E-08 &3.00\\
h/16 &2.25E-09 &3.00 &1.43E-09 &3.00 &2.25E-09 &3.00 &1.43E-09 &3.00\\
\hline
\end{tabular}
\vspace{5pt}
\caption{Accuracy of the $L^2$ projection using $P^2$ approximation.}
\label{tab:1P2}
\end{table}

\noindent{\bf Example 2.}  {\sl Accuracy test of the IRP DG-RK3 scheme solving (\ref{ps})}\\
We apply the IRP DG-RK3 scheme to solve (\ref{ps}), with time step set as $\Delta t=\min (\frac{\Delta x}{3\sigma}, \Delta x^{\frac{2}{3}})$ when using $P^1$ polynomials, and  $\Delta t=\frac{\Delta x}{6\sigma}$  when using $P^2$ polynomials, where 
\begin{align}\label{wavespd}
\sigma =\max \limits_j  \max \left\lbrace \sqrt{-p'(v^-_{j+\frac{1}{2}})},\sqrt{-p'(v^+_{j+\frac{1}{2}})} \right\rbrace.
\end{align}
The numerical results evaluated at $T=0.1$ are given in Table \ref{tab:2P1} and Table \ref{tab:2P2}, where the reference solution is calculated using the fourth order DG scheme on 4096 meshes. These results show that the IRP DG scheme maintains the optimal order of accuracy in $L^1$ norm.
The order of accuracy in $L^{\infty}$ norm is compromised in some cases. 

\begin{table}[htbp]
\centering
\begin{tabular}{|c|cccc||cccc|}
\hline
$P^1$-DG&\multicolumn{4}{|c||}{DG without limiter}  &\multicolumn{4}{|c|}{DG with limiter}
\\ \hline
$\Delta x$ &$L^{\infty}$ error &Order &$L^1$ error & Order &$L^{\infty}$ error &Order &$L^1$ error & Order\\
\hline
h  &1.41E-03 &/    &1.07E-03 &/    &2.54E-03 &/    &1.11E-03 &/\\
h/2  &3.39E-04 &2.05 &2.70E-04 &1.99 &8.92E-04 &1.51 &2.85E-04 &1.96\\
h/4 &8.30E-05 &2.03 &6.59E-05 &2.04 &3.31E-04 &1.43 &6.95E-05 &2.03\\
h/8 &1.98E-05 &2.07 &1.53E-05 &2.10 &8.87E-05 &1.90 &1.66E-05 &2.07\\
h/16 &4.51E-06 &2.14 &3.27E-06 &2.23 &1.35E-05 &2.72 &3.41E-06 &2.28\\
\hline
\end{tabular}
\vspace{5pt}
\caption{Accuracy test with $P^1$-DG scheme for (\ref{ps}). }
\label{tab:2P1}
\end{table}

\begin{table}[htbp]
\centering
\begin{tabular}{|c|cccc||cccc|}
\hline
$P^2$-DG&\multicolumn{4}{|c||}{DG without limiter}  &\multicolumn{4}{|c|}{DG with limiter}
\\ \hline
$\Delta x$ &$L^{\infty}$ error &Order &$L^1$ error & Order &$L^{\infty}$ error &Order &$L^1$ error & Order\\
\hline
h  &2.21E-05 &/    &1.81E-05 &/    &2.79E-04 &/    &3.80E-05 &/\\
h/2  &3.14E-06 &2.82 &2.75E-06 &2.72 &8.85E-05 &1.66 &7.54E-06 &2.33\\
h/4 &3.62E-07 &3.12 &3.13E-07 &3.14 &1.34E-05 &2.72 &8.71E-07 &3.11\\
h/8 &3.87E-08 &3.23 &3.27E-08 &3.26 &1.50E-06 &3.16 &9.84E-08 &3.15\\
h/16 &3.25E-09 &3.57 &2.75E-09 &3.58 &3.01E-07 &2.31 &1.16E-08 &3.09\\
\hline
\end{tabular}
\vspace{5pt}
\caption{Accuracy test with $P^2$-DG scheme for (\ref{ps}).}
\label{tab:2P2}
\end{table}

\noindent{\bf Example 3.}  {\sl Accuracy test of the IRP DG-RK3 scheme solving (\ref{pes})}\\
We apply the IRP DG-RK3 scheme to solve system (\ref{pes}) with $\epsilon =0.01$.  Flux parameters and the time step are chosen as  
\begin{align*}
&  P^1:   \qquad \Delta t =\min\left\{ \frac{\Delta x}{4\sigma},\frac{\Delta x^2}{6\epsilon \beta _0}, \Delta x^{\frac{2}{3}} \right\},  \quad \beta _0=2, \\
 &  P^2:   \qquad \Delta t =\min \left\{ \frac{\Delta x}{12\sigma},\frac{\Delta x^2}{12\epsilon (8\beta _1+\beta _0-2)}\right\}, \quad (\beta _0,\beta _1)=(2,\frac{1}{4}),  
\end{align*}
where $\sigma$ is defined in (\ref{wavespd}). The reference solution is computed by $P^2$-DG using 4096 meshes. 
Both errors and orders of numerical solutions at $T=0.1$ are given  in  Table \ref{tab:3P1} and Table \ref{tab:3P2}, from which we see that  our IRP DG-RK3 scheme maintains the desired order of accuracy.  A noticeable difference between Example 2 and Example 3 is that the accuracy result is better for the latter. One observation that might explain this is that in Example 3, the limiter is imposed only in the first several time steps, yet in Example 2, the limiter is called more frequently in time. \\   


\begin{table}[htbp]
\centering
\begin{tabular}{|c|cccc||cccc|}
\hline
$P^1$-DG&\multicolumn{4}{|c||}{DG without limiter}  &\multicolumn{4}{|c|}{DG with limiter}
\\ \hline
$\Delta x$ &$L^{\infty}$ error &Order &$L^1$ error & Order &$L^{\infty}$ error &Order &$L^1$ error & Order\\
\hline
h  &1.27E-03 &/ &1.03E-03 &/ &1.98E-03 &/ &1.04E-03 &/\\
h/2  &2.94E-04 &2.11 &2.47E-04 &2.05 &3.28E-04 &2.59 &2.48E-04 &2.07\\
h/4 &6.63E-05 &2.15 &5.61E-05 &2.14 &6.63E-05 &2.31 &5.61E-05 &2.14\\
h/8 &1.46E-05 &2.18 &1.19E-05 &2.24 &1.46E-05 &2.18 &1.19E-05 &2.24\\
h/16 &2.86E-06 &2.35 &2.23E-06 &2.42 &2.86E-06 &2.35 &2.23E-06 &2.42\\
\hline
\end{tabular}
\vspace{5pt}
\caption{Accuracy test with $P^1$-DG scheme for (\ref{pes}).}
\label{tab:3P1}
\end{table}

\begin{table}[htbp]
\centering
\begin{tabular}{|c|cccc||cccc|}
\hline
$P^2$-DG&\multicolumn{4}{|c||}{DG without limiter}  &\multicolumn{4}{|c|}{DG with limiter}
\\ \hline
$\Delta x$ &$L^{\infty}$ error &Order &$L^1$ error & Order &$L^{\infty}$ error &Order &$L^1$ error & Order\\
\hline
h    &2.49E-05 &/    &2.41E-05 &/    &2.49E-05 &/    &2.41E-05 &/\\
h/2  &5.26E-06 &2.24 &4.72E-06 &2.35 &5.26E-06 &2.24 &4.72E-06 &2.35\\
h/4  &1.27E-06 &2.05 &9.49E-07 &2.32 &1.27E-06 &2.05 &9.49E-07 &2.32\\
h/8  &2.37E-07 &2.42 &1.74E-07 &2.44 &2.37E-07 &2.42 &1.74E-07 &2.45\\
h/16 &3.64E-08 &2.70 &2.75E-08 &2.66 &3.64E-08 &2.70 &2.75E-08 &2.66\\
h/32 &4.70E-09 &2.96 &3.63E-09 &2.92 &4.70E-09 &2.96 &3.63E-09 &2.92\\
\hline
\end{tabular}
\vspace{5pt}
\caption{Accuracy test with $P^2$-DG scheme for (\ref{pes}).}
\label{tab:3P2}
\end{table}

\noindent{\bf Example 4.} {\sl Convergence of viscous profiles to the entropy solution} \\
In this example, we illustrate the convergence of viscous solutions to the entropy solution when taking $\epsilon \sim O(\Delta x^{r})$ for some $r>0$. 
We apply $P^{k}$-DG to (\ref{pes}), with the reference solution obtained from $P^3$-DG scheme to (\ref{ps}) on 4096 meshes. The time steps are chosen as in Example 3.  From Table \ref{tab:4P1} and Table \ref{tab:4P2} we observe  the optimal convergence as long as  $r\geq k+1$. \\

\begin{table}[htbp]
\centering
\begin{tabular}{|c|cccc||cccc|}
\hline
$P^1$-DG&\multicolumn{4}{|c||}{$\epsilon =\Delta x^2$}  &\multicolumn{4}{|c|}{$\epsilon =\Delta x^3$}
\\ \hline
$\Delta x$ &$L^{\infty}$ error &Order &$L^1$ error & Order &$L^{\infty}$ error &Order &$L^1$ error & Order\\
\hline
h    &9.72E-04 &/    &4.59E-04 &/    &1.24E-03 &/    &8.69E-04 &/\\
h/2  &2.65E-04 &1.88 &1.30E-04 &1.82 &3.29E-04 &1.91 &2.47E-04 &1.81\\
h/4  &6.83E-05 &1.95 &3.65E-05 &1.83 &8.21E-05 &2.00 &6.32E-05 &1.97\\
h/8  &1.70E-05 &2.01 &1.02E-05 &1.83 &1.97E-05 &2.06 &1.50E-05 &2.08\\
h/16 &4.08E-06 &2.06 &2.96E-06 &1.79 &4.50E-06 &2.13 &3.23W-06 &2.22\\
\hline
\end{tabular}
\vspace{5pt}
\caption{Convergence test}
\label{tab:4P1}
\end{table}

\begin{table}[htbp]
\centering
\begin{tabular}{|c|cccc||cccc|}
\hline
$P^2$-DG&\multicolumn{4}{|c||}{$\epsilon =\Delta x^3$}  &\multicolumn{4}{|c|}{$\epsilon =\Delta x^5$}
\\ \hline
$\Delta x$ &$L^{\infty}$ error &Order &$L^1$ error & Order &$L^{\infty}$ error &Order &$L^1$ error & Order\\
\hline
h    &2.60E-04 &/    &1.71E-04 &/    &2.50E-05 &/    &1.95E-05 &/\\
h/2  &3.24E-05 &3.00 &2.17E-05 &2.98 &2.93E-06 &3.09 &2.77E-06 &2.82 \\
h/4  &4.03E-06 &3.01 &2.70E-06 &3.01 &3.55E-07 &3.05 &3.13E-07 &3.15\\
h/8  &5.01E-07 &3.01 &3.38E-07 &3.00 &3.83E-08 &3.21 &3.27E-08 &3.26\\
h/16 &6.19E-08 &3.02 &4.21E-08 &3.00 &3.22E-09 &3.58 &2.73E-09 &3.58\\
\hline
\end{tabular}
\vspace{5pt}
\caption{Convergence test}
\label{tab:4P2}
\end{table}

In the next two examples, we test the IRP DG-RK3 scheme to solve (\ref{ps}) with Riemann initial data. It is known in \cite{Sm94} that for each given state $w_l=(v_l, u_l)$, two shock curves are governed by 
\begin{align}\label{SW}
u=u_l -\sqrt{(v-v_l)(p(v_l)-p(v))},  
\end{align}
and two rarefaction curves are given by 
\begin{align}
u=u_l + \int ^v_{v_l}\sqrt{-p'(y)}dy \label{RF1},  \\
u=u_l - \int ^v_{v_l}\sqrt{-p'(y)}dy \label{RF2}.
\end{align}
From these curves, we identify two cases for our testing.\\

\noindent{\bf Example 5.}  {\sl Shock-rarefaction wave}\\
Consider the following Riemann initial data,  
$$
\mathbf{w}_0(x)=\left\{ \begin{array}{ll} 
(1, 0), & x<0, \\
(\frac{1}{4}, 0.1053), & x>0.
\end{array}
\right.
$$
Following the procedure in \cite[Section A, Chapter 17]{Sm94} we see that such initial data will induce a composite wave, a back shock wave followed by a front rarefaction wave. Therefore, we solve the equations (\ref{SW}) and (\ref{RF2}) for the intermediate state $(\frac{1}{2},-0.9053)$ and                                                                                                                                                                                                                                                                                                                                                                                                                                                                                                                                                                                                                                                                                                                                                                                                                                                                                                                                                                                                                                                                                                                                                                                                                                                                                                                                                                                                                                                                                                                                                                                                                                                                                                                                                                                                                                                                                                                                                                                                                                                                                                                                                                                                                                                                                                    construct the exact solution. The invariant region is given by
$\Sigma=\{(v,u)\big |\quad r(v,u)\leq 0.1053,\; s(v,u)\geq 0.1053\}$. 
We test the $P^1$-DG scheme over $[-1,1]$ at final time $T=0.1$ on 128 meshes. We see in Figure \ref{fig:ShockR} that the oscillations around discontinuities are either damped or oppressed by the IRP limiter. \\

\begin{figure}[htbp]
\centering
\includegraphics[height = 2in]{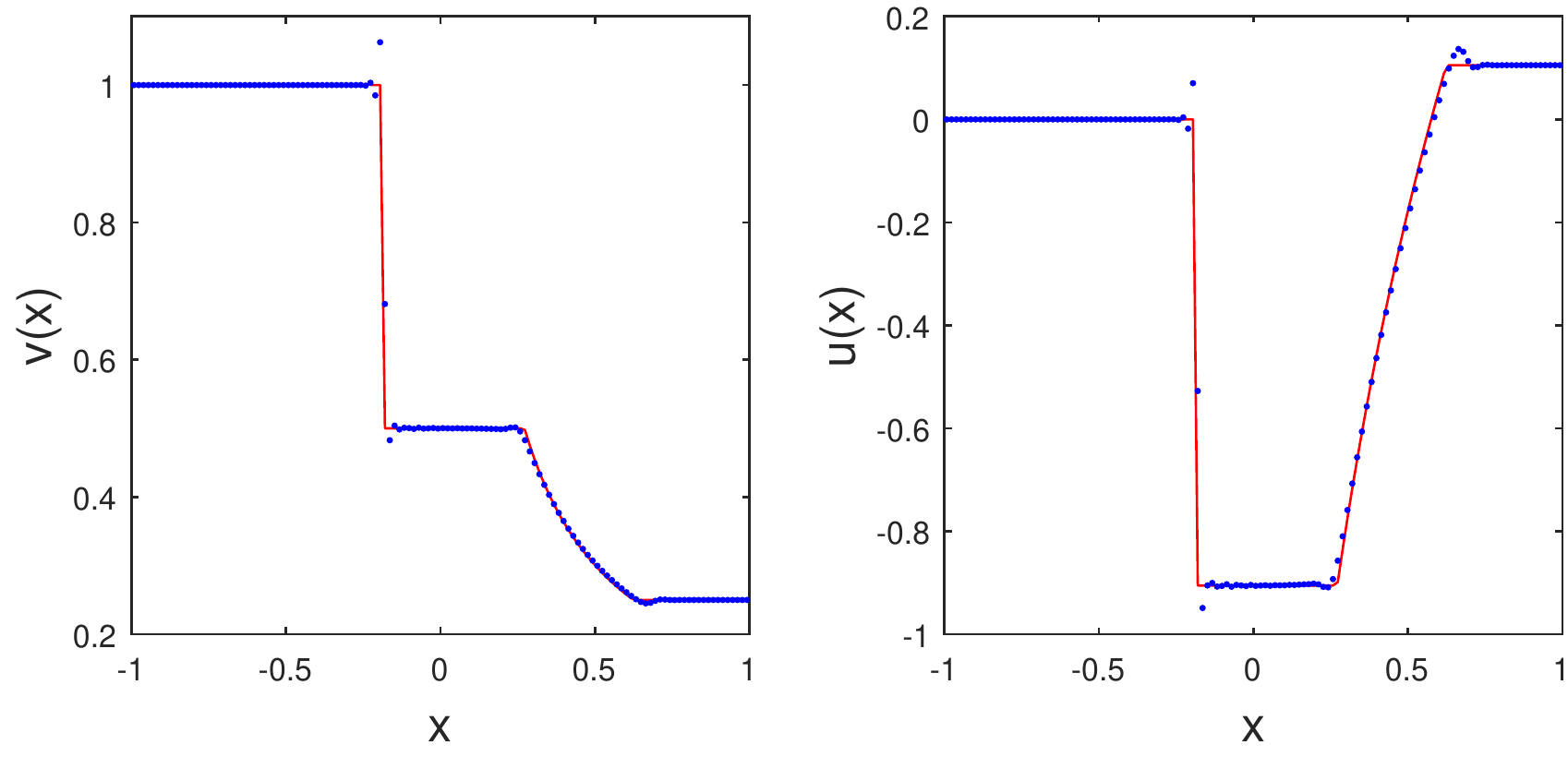}
\includegraphics[height = 2in]{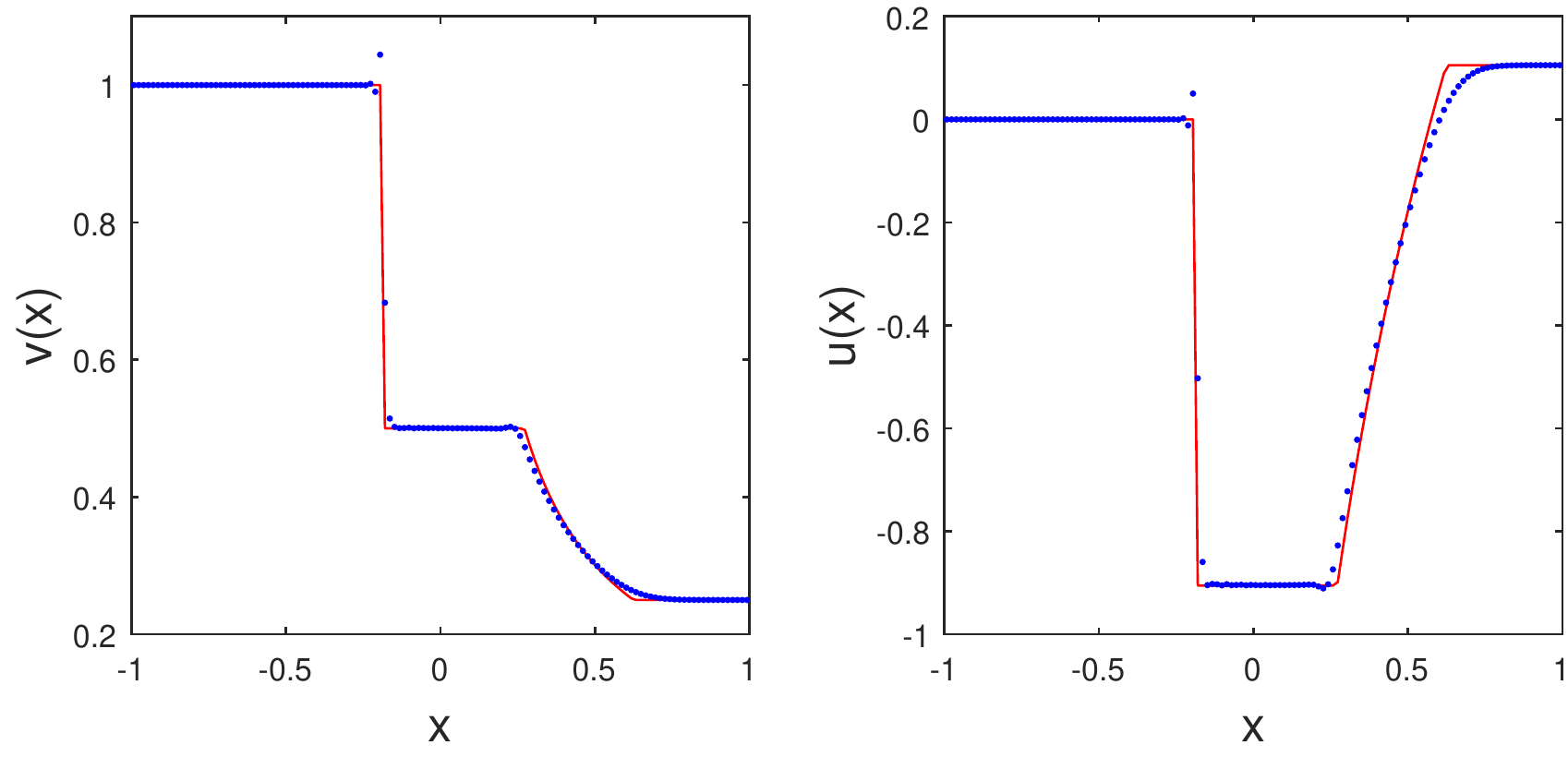}
\caption{Shock-rarefaction wave problem. Exact solution (solid line) vs numerical solution (dots); Top: DG without limiter; Bottom: DG with limiter.}
\label{fig:ShockR}
\end{figure}

\noindent{\bf Example 6.}  {\sl Rarefaction-shock wave}\\  Consider the following Riemann initial data,  
$$
\mathbf{w}_0(x)=\left\{ \begin{array}{ll} 
(1, 0), & x<0,  \\
(2,  -0.3509), & x>0.
\end{array}
\right.
$$
A similar checking tells that such initial data will induce a composite wave, a back rarefaction wave followed by a front shock wave. From (\ref{SW}) and (\ref{RF1}) we find the intermediate state $(1.2,0.2118)$ and construct the exact solution. The invariant region is given by $\Sigma=\{(v,u)\big | \quad r(v,u)\leq 0, \; s(v,u)\geq 0\}$. We test the $P^1$-DG scheme over $[-1,1]$ at final time $T=0.1$ on 128 meshes. It shows from Figure \ref{fig:RarefS} that oscillations occur near discontinuities. However, after the limiter is applied, the oscillations are reduced.

\begin{figure}[htbp]
\centering
\includegraphics[height = 2in]{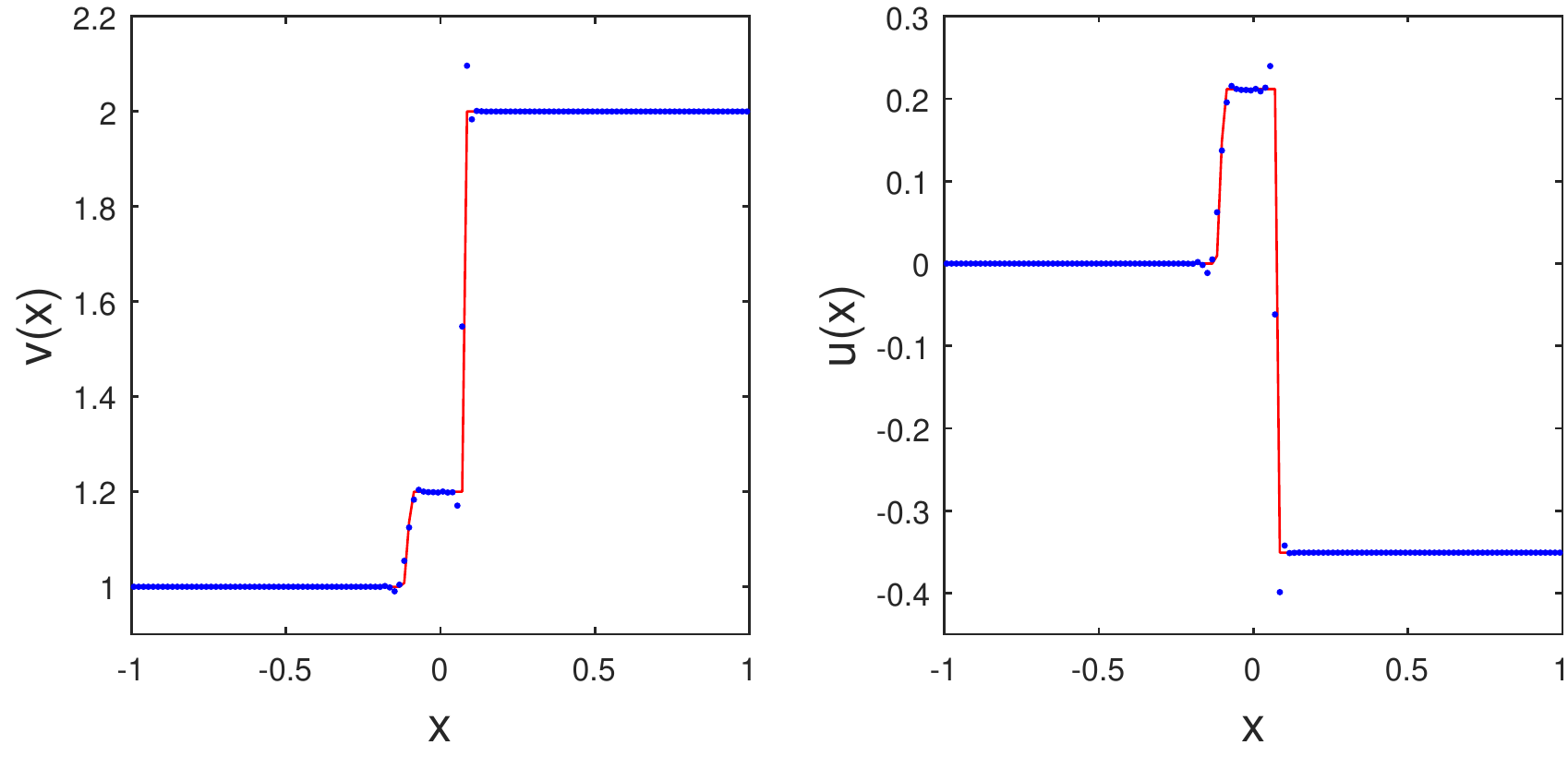}
\includegraphics[height = 2in]{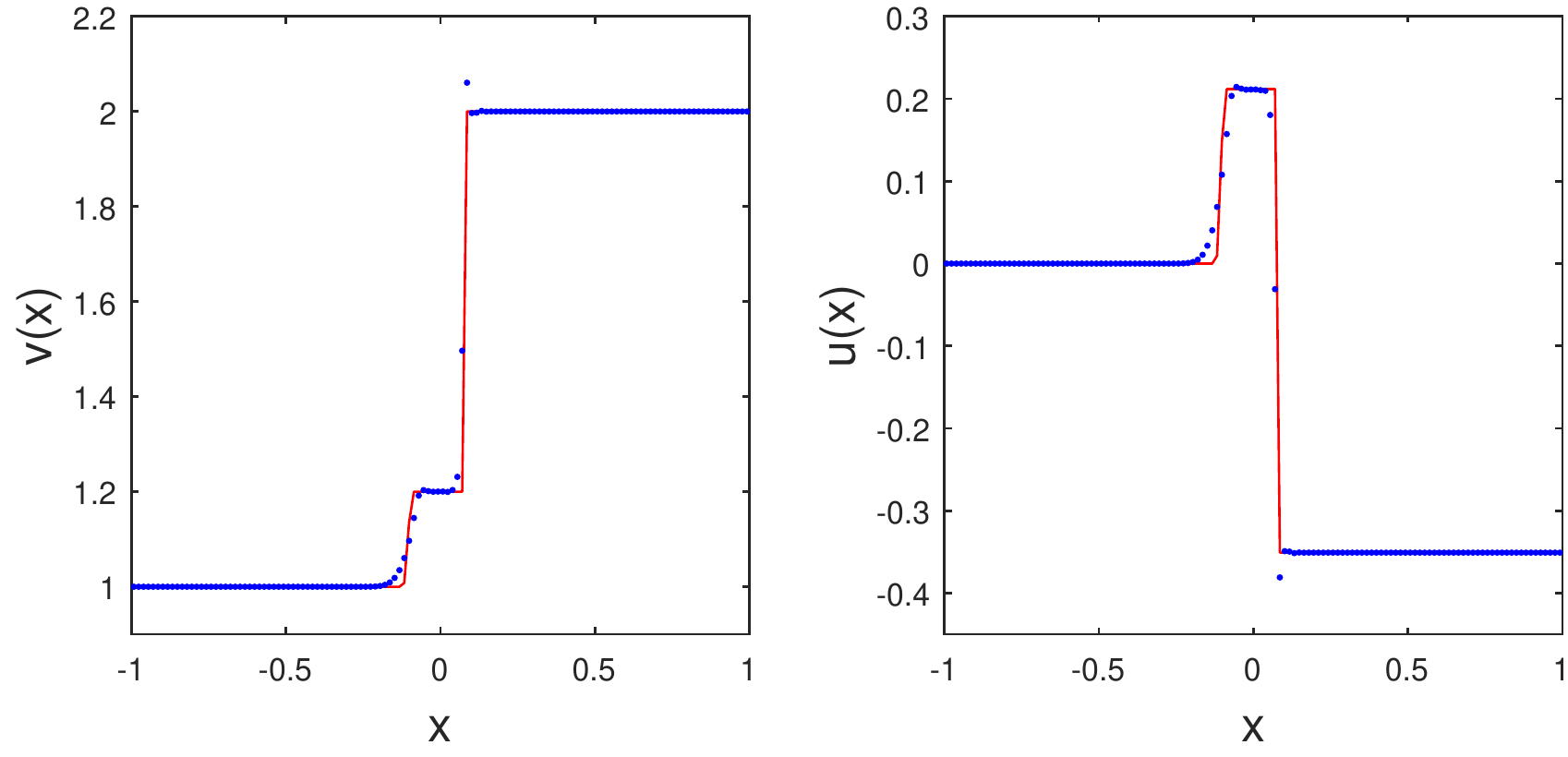}
\caption{Rarefaction-shock wave problem. Exact solution (solid line) vs numerical solution (dots); Top: DG without limiter; Bottom: DG with limiter.}
\label{fig:RarefS}
\end{figure}

\begin{rem}The IRP limiter presented in this work is a mild limiter. As we can see in Example 5 and 6 that not all oscillations can be completely damped, even though the invariant region is preserved. For big oscillations, some stronger limiters might be needed. For example, the weighted essentially non-oscillatory (WENO) limiter developed by Zhong and Shu in \cite{ZhongShu}. 
\end{rem}

\section{Conclusions remarks}
In this paper, we introduced an explicit IRP limiter for the p-system, including the isentropic gas dynamic equation, and the viscous p-system. The limiter itself is shown to preserve the accuracy of high order approximation in general cases through both rigorous analysis and numerical tests. We also specify sufficient conditions, including CFL conditions and test sets, for high order IRP DG schemes with Euler forward time discretization. Numerical tests on such schemes with RK3 time discretization confirm the desired properties. An interesting observation is that the IRP DG scheme solving the viscous p-system is more accurate  than the one solving the p-system.  For the latter the IRP limiter is called much more frequently,  indicating a possible error accumulation through the evolution in time. 


\appendix
\numberwithin{equation}{section}

\section{Proofs of Lemma \ref{lem:p1} and \ref{lem:p2}}\label{appsec:lems}
\noindent{\sl \textbf{Proof of Lemma \ref{lem:p1}}}:  Set 
\begin{align*}
\mathbf{p}_j(\xi) & =\mathbf{w}^n_h(x_j+\frac{\Delta x}{2}\xi)=\mathbf{w}^n_h(x)\big |_{I_j}, \quad \xi\in [-1, 1],
\end{align*} 
then for any $\gamma \in [-1,1] \setminus \{0\}$, we have 
\begin{align*}
\mathbf{p}_j(\xi)=\frac{1}{2}\left(1 - \frac{\xi}{\gamma}\right)\mathbf{p}_j(-\gamma)+\frac{1}{2}\left(1+ \frac{\xi}{\gamma}\right) \mathbf{p}_j(\gamma),
\end{align*}
from which the cell average can be expressed as  
$\bar{\mathbf{w}}^n_j=\frac{1}{2}\mathbf{p}_j(-\gamma)+\frac{1}{2}\mathbf{p}_j(\gamma)$.  
A direct calculation of ${D}_j$ in (\ref{split}) using flux (\ref{df}) gives
\begin{equation}\label{appd:DP1} 
\begin{aligned}
{D}_j=&2\epsilon \mu \alpha _1\mathbf{p}_{j+1}(-\gamma)+2\epsilon \mu \mathbf{p}_{j+1}(\gamma)+2\epsilon \mu\alpha _2\mathbf{p}_{j-1}(-\gamma)+2\epsilon \mu \alpha _1\mathbf{p}_{j-1}(\gamma)\\
&+\left(\frac{1}{2}-2\epsilon\mu (\alpha _1+\alpha _2)\right)\mathbf{p}_j(-\gamma)+\left(\frac{1}{2}-2\epsilon \mu (\alpha _1+\alpha _2)\right)\mathbf{p}_j(\gamma),
\end{aligned}
\end{equation}
where
\begin{align*}
\alpha _1=\frac{\beta _0}{2}+\frac{\beta _0-1}{2\gamma}, \quad \alpha _2=\frac{\beta _0}{2}-\frac{\beta _0-1}{2\gamma}
\end{align*}
and both are non-negative with $\beta _0\geq \frac{1}{2}$ and (\ref{p1gamma}). Note that $\quad \alpha_1+\alpha_2=\beta_0$. If $\mu\leq \frac{1}{4\epsilon \beta _0}$,  (\ref{appd:DP1})  is a convex combination of vector values over set $S^D_i(i=j-1, j, j+1)$, therefore $D_j\in\Sigma$ due to the convexity of $\Sigma$. Moreover, for non-trivial polynomials, a similar argument to the proof of Lemma  \ref{lem:2} shows that  $D_j\in\Sigma_0$.
\hfill\ensuremath{\square}

\hspace{5pt}

\noindent{\sl \textbf{Proof of Lemma \ref{lem:p2}}}:  Using the same notation for $\mathbf{p}_j(\xi)$ as in the proof of Lemma \ref{lem:p1}, we see that 
 in the case of  $\mathbf{p}_j(\xi)\in \mathbb{P}^2$, for any $\gamma \in (-1,1)$,  
 \begin{align*}
\mathbf{p}_j(\xi)=\frac{(\xi-1)(\xi -\gamma)}{2(1+\gamma)}\mathbf{p}_j(-1)+\frac{(\xi -1)(\xi +1)}{(\gamma -1)(\gamma +1)}\mathbf{p}_j(\gamma)+\frac{(\xi +1)(\xi -\gamma)}{2(1-\gamma)}\mathbf{p}_j(1).
\end{align*}
Then the cell average is 
\begin{align*}
\bar{\mathbf{w}}^n_j=c_1\mathbf{p}_j(-1)+c_2\mathbf{p}_j(\gamma)+c_3\mathbf{p}_j(1),
\end{align*}
where
\begin{align*}
c_1=\frac{1+3\gamma}{6(1+\gamma)}, \quad c_2=\frac{2}{3(1-\gamma^2)}, \quad c_3=\frac{1-3\gamma}{6(1-\gamma)}
\end{align*}
are all non-negative for $|\gamma |\leq \frac{1}{3}$. 
A direct calculation gives ${D}_j$ as
\begin{equation}\label{Dcoe}
\begin{aligned}
{D}_j=&\alpha _4\mathbf{p}_j(-1)+\alpha _5\mathbf{p}_j(\gamma)+\alpha _6\mathbf{p}_j(1)\\
&+2\epsilon \mu \alpha _1(\gamma)\mathbf{p}_{j+1}(-1)+2\epsilon \mu \alpha _2(\gamma)\mathbf{p}_{j+1}(\gamma)+2\epsilon \mu \alpha _3(\gamma)\mathbf{p}_{j+1}(1)\\
&+2\epsilon \mu \alpha _3(-\gamma)\mathbf{p}_{j-1}(-1)+2\epsilon \mu \alpha _2(-\gamma)\mathbf{p}_{j-1}(-\gamma)+2\epsilon \mu \alpha _1(-\gamma)\mathbf{p}_{j-1}(1)
\end{aligned}
\end{equation}
with
\begin{align*}
&\alpha _1(\gamma)=\beta _0+\frac{8\beta _1-3-\gamma}{2(\gamma +1)}, \quad \alpha _2(\gamma)=\frac{8\beta _1-2}{\gamma ^2-1}, \quad \alpha _3(\gamma)=\frac{8\beta _1-1-\gamma}{2(1-\gamma)},\\
&\alpha _4=\frac{1+3\gamma}{6(1+\gamma)}-2\epsilon \mu (\alpha _3(-\gamma)+\alpha _1(\gamma)), \\
&\alpha _5=\frac{2}{3(1-\gamma^2)}-2\epsilon \mu (\alpha _2(-\gamma)+\alpha _2(\gamma)), \\
&\alpha _6=\frac{1-3\gamma}{6(1-\gamma)}-2\epsilon \mu (\alpha _1(-\gamma)+\alpha _3(\gamma)).
\end{align*}
Note that (\ref{Dcoe}) is a convex combination of vector values over set $S^D_j$  if $\alpha_i \geq 0$ for $i=1,\cdots ,6$. This is guaranteed by (\ref{p2beta}) and $|\gamma|\leq 8\beta _1 -1$, together with $\mu \leq \mu _0$ given in (\ref{mu0}).
\hfill\ensuremath{\square}

\section{ Is $C_4$ uniformly bounded?}\label{appsec:C4}
In this appendix, we give two examples on the magnitude of  $C_4$, which is defined in (\ref{c4}) in the proof of Lemma \ref{lem:5}. 
\\

\noindent{\bf Example 1 ( $C_4$ is uniformly bounded)}  \\
Consider $v(x)=1+h^2x^2$, $u(x)=1$ for $x\in[0,1]$, where $0<h\ll 1$.
For $\gamma=3$, we have 
\begin{align*}
r(v(x),u(x))=1+\sqrt{3}\left(-1+\frac{1}{1+h^2x^2}\right), \quad s(v(x),u(x))=1-\sqrt{3}\left(-1+\frac{1}{1+h^2x^2}\right),
\end{align*}
and $r_0=s_0=1$.

Now, consider linear interpolation of $v(x)$ and $u(x)$ over $[0,1]$ as follows
\begin{align*}
\tilde{v}(x)=1-\frac{1}{8}h^2+\frac{3}{4}h^2x, \quad \tilde{u}(x)=1,
\end{align*}
where the interpolation points are chosen as $\frac{1}{4}$ and $\frac{1}{2}$. The corresponding cell averages of two polynomials are
\begin{align*}
\bar{\tilde{v}}=1+\frac{1}{4}h^2, \quad \bar{\tilde{u}}=1.
\end{align*}
Also, we can find
\begin{align*}
& r(\tilde{v}(x),\tilde{u}(x))=1+\sqrt{3}\left(-1+\frac{1}{1-\frac{h^2}{8}+\frac{3h^2}{4}x}\right), \\
&  s(\tilde{v}(x),\tilde{u}(x))=1-\sqrt{3}\left(-1+\frac{1}{1-\frac{h^2}{8}+\frac{3h^2}{4}x}\right).
\end{align*}
We can verify that $ r(\tilde{v}(x),\tilde{u}(x))>1$, $ s(\tilde{v}(x),\tilde{u}(x))<1$ for $x <\frac{1}{6}$;  namely, $\mathbf{q}(x)=(\tilde{v}(x),\tilde{u}(x))$ lies outside of invariant region.

Let $\bar{r}=r(\bar{\tilde{v}},\bar{\tilde{u}})$, $\bar{s}=s(\bar{\tilde{v}},\bar{\tilde{u}})$, then we have
\begin{align*}
r_0-\bar{r}=-\sqrt{3}\left(-1+\frac{1}{1+\frac{h^2}{4}}\right), \quad \bar{s}-s_0=-\sqrt{3}\left(-1+\frac{1}{1+\frac{h^2}{4}}\right),
\end{align*}
which indicates $C_4=1$.\\

\noindent{\bf Example 2. ($C_4=O(h^{-2})$)} \\
Consider $v(x)=\frac{2\sqrt{3}}{2\sqrt{3}+(1+h^2)2(x-1)}$, $u(x)=(1-h^2)x+\frac{h^2-1}{2}$ for $x\in[0,1]$, where $0<h\ll 1$. For $\gamma =3$, we have
\begin{align*}
r(v(x),u(x))=\frac{-3-h^2}{2}+2x,\quad s(v(x),u(x))=\frac{1+3h^2}{2}-2h^2x,
\end{align*}
and $r_0=s_0=\frac{1-h^2}{2}$.

Now, consider linear interpolation of $v(x)$ and $u(x)$ over $[0,1]$ as follows
\begin{align*}
\tilde{v}(x)=&2\sqrt{3}\left(\frac{1}{1-2\sqrt{3}+h^2}-\frac{4}{3-4\sqrt{3}+3h^2}\right)+\frac{-8\sqrt{3}(1+h^2)x}{(-3+4\sqrt{3}-3h^2)(-1+2\sqrt{3}-h^2)}\\
\tilde{u}(x)=&\frac{h^2-1}{2}+(1-h^2)x,
\end{align*}
where the interpolation points are chosen as $\frac{1}{4}$ and $\frac{1}{2}$. The corresponding averages of two polynomials are
\begin{align*}
\bar{\tilde{v}}=\frac{2\sqrt{3}}{-1+2\sqrt{3}-h^2}, \quad \bar{\tilde{u}}=0.
\end{align*}
Let $\bar{r}=r(\bar{\tilde{v}},\bar{\tilde{u}})$, $\bar{s}=s(\bar{\tilde{v}},\bar{\tilde{u}})$, then we have
\begin{align*}
\bar{r}=\frac{-1-h^2}{2}, \quad \bar{s}=\frac{1+h^2}{2}
\end{align*}
and
\begin{align*}
r_0-\bar{r}=1, \quad \bar{s}-s_0=h^2.
\end{align*}
In the following table, we show that with certain choices of $h$, we have $\theta_2<\theta_1<1$, which indicates $C_4=\frac{1}{h^2}$, where $\theta _1$ and $\theta _2$ are as defined in \eqref{limiter}, 
\begin{table}[htbp]
\centering
\begin{tabular}{|c||c|c|c|c||c|c|c|c|}
\hline
h &$r_{\max}$ &$r_0$ &$\bar{r}$ &$\theta_1$ &$s_{\min}$ &$s_0$ &$\bar{s}$ &$\theta_2$\\
\hline
0.5  &3.831 &0.375 &-0.625 &0.224 &-3.081 &0.375 &0.625 &0.068\\
0.1  &1.310 &0.495 &-0.505 &0.551 &-0.320 &0.495 &0.505 &0.012\\
0.01 &1.27823 &0.49995 &-0.50005 &0.562234 &-0.27833 &0.49995 &0.50005 &0.00013\\
\hline
\end{tabular}
\end{table}


\section*{Acknowledgements}
The authors would like to thank Professor Chi-Wang Shu for helpful discussions on the subtlety of the limiter in the system case, which led to our construction of two examples in Appendix B. This research was supported by the National Science Foundation under Grant DMS1312636.
     
\bigskip

\bibliographystyle{plain}

\end{document}